%% file: main.tex
\documentclass[oneside,11pt]{amsart}
\usepackage{amsmath}
\usepackage{amssymb}
\usepackage{graphicx}
\usepackage{tikz}
\usepackage{color}
\newtheorem{thm}{Theorem}[section]
\newtheorem{prop}[thm]{Proposition}
\newtheorem{lem}[thm]{Lemma}
\newtheorem{cor}[thm]{Corollary}

\theoremstyle{definition}
\newtheorem{defn}[thm]{Definition}
\newtheorem{rem}[thm]{Remark}

\newtheorem{prob}[thm]{Problem}
\newtheorem{cons}[thm]{Construction}
\newtheorem{setup}[thm]{Set up}

\newcommand{\abs}[1]{\lvert{#1}\rvert}
\newcommand{\bigabs}[1]{\bigl|{#1}\bigr|}

\renewcommand{\bar}[1]{\overline{#1}}

\newcommand{\bigset}[2]{ \bigl\{ \, {#1} \bigm| {#2} \, \bigr\} }
\renewcommand{\emptyset}{\varnothing}

\renewcommand{\setminus}{-}

\newcommand{\field}[1]{\mathbb{#1}}
\newcommand{\Z}{\field{Z}}
\newcommand{\R}{\field{R}}
\newcommand{\Q}{\field{Q}}
\newcommand{\N}{\field{N}}

\newcommand{\PP}{\field{P}}

\DeclareMathOperator{\CAT}{CAT}

\usepackage{ifthen}

\newcommand{\showcomments}{yes}

\newsavebox{\commentbox}
{\ifthenelse{\equal{\showcomments}{yes}}%
{\footnotemark
    \begin{lrbox}{\commentbox}
    \begin{minipage}[t]{1.25in}\raggedright\sffamily\upshape\tiny
    \footnotemark[\arabic{footnote}]}%
{\begin{lrbox}{\commentbox}}}%
{\ifthenelse{\equal{\showcomments}{yes}}%
{\end{minipage}\end{lrbox}\marginpar{\usebox{\commentbox}}}%
{\end{lrbox}}}

\begin{document}

\title{Distortion of surfaces in $3$--manifolds}

\author{Hoang Thanh Nguyen}
\address{Department of Mathematical Sciences\\
University of Wisconsin--Milwaukee\\
P.O.~Box 413\\
Milwaukee, WI 53201\\
USA}
\email{nguyen36@uwm.edu}

\date{\today}

\begin{abstract}
Let $g \colon S \looparrowright N$ be a properly immersed $\pi_1$--injective surface in a non-geometric $3$--manifold $N$.
We compute the distortion of $\pi_1(S)$ in $\pi_1(N)$ and show that how it is related to separability of $\pi_1(S)$ in $\pi_1(N)$. The only possibility of the distortion is linear, quadratic, exponential, and double exponential.
\end{abstract}


\maketitle

\section{Introduction}
\label{sec:Introduction}

In geometric group theory, the distortion of a finitely generated subgroup $H$ in a finitely generated subgroup $G$ is a classical notion. Let $\mathcal{S}$ and $\mathcal{A}$ be finite generating sets of $G$ and $H$ respectively. The subgroup $H$ itself admits a word length metric, but it also inherits an induced metric from the group $G$. The distortion of $H$ in $G$ compares these metrics on $H$. In other words, we would like to know how the inclusion $H \hookrightarrow G$ preserves geometric properties of $H$. More precisely, the \emph{distortion} of $H$ in $G$ is the function
\[
   \Delta_{H}^G(n)
     = \max \bigset{\abs{h}_{\mathcal{A}}}{\text{$h\in{H}$ and $\abs{h}_{\mathcal{S}}\leq{n}$} }
\]
Up to a natural equivalence, the function $\Delta_{H}^{G}$ does not depend on the choice of finite generating sets $\mathcal{S}$ and $\mathcal{A}$.

This paper is devoted to understanding the large scale geometry of immersed surfaces in $3$--manifolds by using distortion of the surface group. In fact, the purpose is to address the following problem:
\begin{prob}
\label{prob:Wise}
Let $S \looparrowright N$ be a properly immersed $\pi_1$--injective surface in a $3$--manifold $N$. What is the distortion of $\pi_1(S)$ in $\pi_1(N)$? How does it relate to algebraic properties of $\pi_1(S) \le \pi_1(N)$, topological properties of the immersion and geometries of components in the JSJ decomposition? 
\end{prob}

Dani Wise observed that Problem~\ref{prob:Wise} is important in the study of cubulations of $3$--manifold groups. The goal of cubulation is to find a suitable collection of immersed surfaces and then study the action of the fundamental group of the $3$--manifold on the $\CAT(0)$ cube complex dual to the collection of immersed surfaces. Whenever the fundamental group acts properly and cocompactly, surfaces must be undistorted. 

A compact, orientable, irreducible $3$--manifold $N$ with empty or toroidal boundary is \emph{geometric} if its interior admits a geometric structure in the sense of Thurston. The answer to Problem~\ref{prob:Wise} is relatively well-understood in the geometric case. By Hass \cite{Hass87}, if $N$ is a Seifert fibered space then up to homotopy, the surface $S$ is either vertical (i.e, union of fibers) or horizontal (i.e, tranverses to fibers). In either case, $\pi_1(S)$ is undistorted in $\pi_1(N)$. If $N$ is a hyperbolic $3$--manifold, then by Bonahon-Thurston (\cite{Bonahon86}, \cite{Thurston79}) the distortion is linear when the surface is geometrically finite and the distortion is exponential when the surface is geometrically infinite.

By Geometrization Theorem, a non-geometric $3$--manifold can be cut into hyperbolic and Seifert fibered  ``blocks''  along a JSJ decomposition. It is called a \emph{graph manifold} if all the blocks are Seifert fibered, otherwise it is a \emph{mixed manifold}. 

An immersed surface $S$ in a non-geometric manifold $N$ is called \emph{properly immersed} if the preimage of $\partial N$ under the immersion is $\partial S$. Roughly speaking, if the surface $S$ is properly immersed $\pi_1$--injective in the non-geometric manifold $N$ then up to homotopy, the JSJ decomposition in the manifold into blocks induces a decomposition on the surface into \emph{pieces}. Each piece is carried in either a hyperbolic or Seifert fibered block. A piece in a Seifert fibered block is either vertical  or horizontal, and a piece in a hyperbolic block is either geometrically finite or geometrically infinite.
Yi Liu \cite{YiLiu2017} and Hongbin Sun \cite{Sun18} show that all information about virtual embedding can be obtained by examining  the \emph{almost fiber part} $\Phi(S)$, that is, the union of horizontal and geometrically infinite pieces. We remark that virtual embedding is equivalent to subgroup separability \cite{Scott78}, \cite{PrzytyckiWise14} (a subgroup $H \le G$ is called \emph{separable} if for any $g \in G-H$ there exists a finite index subgroup $K \le G$ such that $H \le K$ and $g \notin K$).

The following theorem is the main theorem in this paper which give a complete answer to Problem~\ref{prob:Wise}. The theorem states for ``clean surfaces'' which we discuss below, but we emphasize here that up to homotopy every properly immersed surface is also a clean surface.

\begin{thm}[Distortion of surfaces in non-geometric $3$--manifolds]
\label{thm:thesismain}
Let $g \colon S \looparrowright N$ be a clean surface in a non-geometric $3$--manifold $N$. Suppose that all Seifert fibered blocks of $N$ are non-elementary. Let $\Delta$ be the distortion of $\pi_1(S)$ in $\pi_1(N)$. There are four mutually exclusive cases:
\begin{enumerate}
    \item
    \label{item:thmmain1}
    If there is a component $S'$ of the almost fiber $\Phi(S)$ such that $S'$ contains a geometrically infinite piece and $\pi_1(S')$ is non-separable in $\pi_1(N)$ then $\Delta$ is double exponential.
\item
\label{item:thmmain2}
Suppose that $\phi(S)$ has no component satisfying (\ref{item:thmmain1}). If there is a component $S'$ of the almost fiber $\Phi(S)$ such that $S'$ contains a geometrically infinite piece then $\Delta$ is exponential.
\item 
Suppose that $\phi(S)$ has no component satisfying (\ref{item:thmmain1}) and (\ref{item:thmmain2}) (i.e, no component of $\phi(S)$ contains a geometrically infinite piece). If there is a component of the almost fiber $\Phi(S)$ containing two adjacent pieces then $\Delta$ is exponential if $\pi_1(S)$ is non-separable in $\pi_1(N)$ and $\Delta$ is quadratic if $\pi_1(S)$ is separable in $\pi_1(N)$.
\item In all other cases, $\Delta$ is linear.
\end{enumerate}
\end{thm}

We note that Theorem~\ref{thm:thesismain} generalizes the main theorem of Hruska-Nguyen in \cite{Hruska-Nguyen}. In the setting of a properly immersed surface in a graph manifold, Hruska and the author (\cite{Hruska-Nguyen}) show that when the surface is an almost fiber, i.e, horizontal, its distortion is always nontrivial. The distortion is quadratic if the fundamental group of the surface is separable in the fundamental group of the manifold and the distortion is exponential otherwise. 

For the definition of nonelementary Seifert fibered space, we refer the reader to Section~\ref{sec:manifolds}.  We note that the properly immersed condition of a surface is not general enough for the purpose of this paper since the almost fiber part $\Phi(S)$ is no longer a properly immersed surface in the $3$--manifold (in fact, it is typical that a boundary circle of the almost fiber part is mapped into a JSJ torus of $N$). We thus introduce the notion of \emph{clean surfaces} which generalizes the notion of properly immersed surfaces by allowing some boundary circles to be mapped into JSJ tori (see Definition~\ref{defn:clean surface}). Clean surfaces are general enough for the purpose of computing distortion in this paper (as the almost fiber part of a clean surface is again a clean surface and a properly immersed $\pi_1$--injective surface is also a clean surface).


We prove Theorem~\ref{thm:thesismain} by using the following strategy. 
We prove that the distortion of a clean surface $S$ in a non-geometric $3$--manifold $N$ depends only on the almost fiber part $\Phi(S)$ (see Theorem~\ref{thm:intro:distortionalmostfiber}) and then we compute the distortion of components of the almost fiber part $\Phi(S)$ in the manifold $N$ (see Theorem~\ref{thm:introductionUpperBound} and Theorem~\ref{thm:intro:graphmanifold}).

\begin{thm}
\label{thm:intro:distortionalmostfiber}
Let $g \colon S \looparrowright N$ be a clean surface in a non-geometric $3$--manifold $N$. We assume that every Seifert fibered block in $N$ is nonelementary. For each component $S_{i}$ of $\Phi(S)$, let $\delta_{S_{i}}$ be the distortion of $\pi_1(S_{i})$ in $\pi_1(N)$. Then the distortion of $H = \pi_1(S)$ in $G= \pi_1(N)$ satisfies
\[
f \preceq \Delta_{H}^{G} \preceq \bar{f}
\]
where
\[
f(n) := \max \bigset{\delta_{S_{i}}(n)}{\text {$S_{i}$ is a component of $\Phi(S)$}}
\] and $\bar{f}$ is the superadditive closure of $f$.
\end{thm}
For the definition of superadditive closure function, we refer the reader to Section~\ref{sec:preli}. We remark that a similar result was proved by Hruska for relatively hyperbolic groups (see Theorem~1.4 in \cite{Hruska2010}), but the conclusion here is stronger because in many cases $\pi_1(S)$ and $\pi_1(N)$ don't satisfy the hypothesis in Theorem~1.4 \cite{Hruska2010}. 

 In \cite{RW98}, Rubinstein-Wang introduce a combinatorial invariant called ``spirality'' and show that it is the obstruction to separability for horizontal surfaces in graph manifolds. Recently, Liu \cite{YiLiu2017} generalizes the work of Rubinstein--Wang to closed surfaces in closed non-geometric $3$--manifolds and Sun \cite{Sun18} generalizes the work of Liu to arbitrary finitely generated subgroups in arbitrary non-geometric $3$--manifolds.
In the setting of a clean almost fiber surface in a graph manifold, the following theorem follows immediately from the work of Hruska--Nguyen \cite{Hruska-Nguyen} and the theorems of Liu \cite{YiLiu2017} and Sun \cite{Sun18}.

\begin{thm}
\label{thm:intro:graphmanifold}
Let $S \looparrowright N$ be a clean almost fiber surface (i.e, $\Phi(S) =S$) in a graph manifold $N$. We assume that all Seifert fibered blocks of $N$ is non-elementary. Let $\Delta$ be the distortion of $\pi_1(S)$ in $\pi_1(N)$. Then
\begin{enumerate}
    \item $\Delta$ is linear if each component of the almost fiber part contains only one horizontal piece.
    \item
    Otherwise, $\Delta$ is quadratic if $\pi_1(S)$ is separable in $\pi_1(N)$, and exponential if $\pi_1(S)$ is non-separable in $\pi_1(N)$.
\end{enumerate}
\end{thm}

To give a complete proof to Theorem~\ref{thm:thesismain}, it remains to compute the distortion of a clean almost fiber surface in a mixed manifold (see Theorem~\ref{thm:introductionUpperBound}). This computation is one of the main components of this paper. We note that a fibered $3$--manifold can be expressed as a mapping torus for a diffeomorphism of the fiber surface. The strategy in the proof of Theorem~\ref{thm:introductionUpperBound} is inspired from Hruska-Nguyen \cite{Hruska-Nguyen} and Woodhouse \cite{Woodhouse16}. However the techniques are different because unlike the setting of a Seifert block where the diffeomorphism of the fiber surface is trivial and the distortion of the fiber surface in the Seifert block is linear, the diffeomorphism of the fiber surface in a hyperbolic block is pseudo-Anosov and the distortion of the fiber surface in the hyperbolic block is exponential. In addition, the generalized definition of spirality by Liu and Sun in a mixed manifold is more elaborate. We use the generalization of Liu and Sun to compute the distortion and show that the distortion is determined by separability of the surface subgroup.

\begin{thm}
\label{thm:introductionUpperBound}
Let $S \looparrowright N$ be a clean almost fiber surface (i.e, $\Phi(S) =S$) in a mixed manifold $N$. We assume that all Seifert fibered blocks of $N$ is non-elementary. Suppose that $S$ contains at least one geometrically infinite piece. Then the distortion of $\pi_1(S)$ in $\pi_1(N)$ is exponential if $\pi_1(S)$ is separable in $\pi_1(N)$, and double exponential if $\pi_1(S)$ is non-separable in $\pi_1(N)$.
\end{thm}


As mentioned above, the strategy for constructing an action of $\pi_1(N)$ on a $\CAT(0)$ cube complex is to find a suitable collection of immersed surfaces and then consider the $\CAT(0)$ cube complex dual to this collection of surfaces. According to Hagen--Przytycki \cite{HagenPrzytycki15} and Tidmore \cite{Tidmore16}
the fundamental groups of chargeless
graph manifolds and chargeless mixed manifolds act cocompactly on $\CAT(0)$ cube complexes. The cubulations constructed by them are each dual to a collection of immersed surfaces, none of which contains a geometrically infinite piece or two adjacent horizontal pieces. It is clear from the corollary below that the cocompact cubulations of Hagen--Przytycki and Tidmore are canonical. For the purpose of obtaining a proper, cocompact cubulation, all surface subgroups must be of the type used by Hagen--Przytycki and Tidmore.

\begin{cor}
Let $G$ be the fundamental group of a non-geometric $3$--manifold. Let $\{H_1,H_2,\dots,H_k\}$ be a collection of codimension--$1$ subgroups of $G$. Let $X$ be the corresponding dual $\CAT(0)$ cube complex. If at least one $H_i$ is the fundamental group of a surface containing two adjacent horizontal pieces or a geometrically infinite piece, then the action of $G$ on $X$ is not proper and cocompact.
\end{cor}

\subsection{Overview}
In Section~\ref{sec:preli} we review some concepts in geometric group
theory. Section~\ref{sec:manifolds} is a review background
about $3$--manifolds and introduced the notion of clean surface. In Section~\ref{sec:almostfiber}, we give the proof of Theorem~\ref{thm:intro:distortionalmostfiber}. The proof of Theorem~\ref{thm:introductionUpperBound} is given in  Section~\ref{sec:distortion mixed}. In Section~\ref{section:putting results}, we discuss about Theorem~\ref{thm:intro:graphmanifold} and Theorem~\ref{thm:thesismain} by combining previous results.

\subsection{Acknowledgements} 
I would like to thank my adviser Chris Hruska for all his help and advice throughout this paper. I would also like to thank Yi Liu, Hongbin Sun, Dan Margalit, Prayagdeep Parija, and Daniel Gulbrandsen for helpful conversations. The author is grateful for the insightful and detailed critiques of the referee that have helped improve the exposition of this paper. The author especially appreciates the referee for pointing out a mistake in Section 5.2 in the earlier version.

\section{preliminaries}
\label{sec:preli}
In this section, we review some concepts in geometric group theory.

Let $(X,d)$ be a metric space, and $\gamma$ a path in $X$. We denote the length of $\gamma$ by $\abs{\gamma}$.

\begin{defn}
\label{def:equivalentfunction}
Let $\mathcal{F}$ be the collection of all functions from positive reals to positive reals. Let $f$ and $g$ be arbitrary elements of $\mathcal{F}$. The function $f$ is \emph{dominated} by a function $g$, denoted by
\emph{$f\preceq g$}, if there are positive constants $A$, $B$, $C$, $D$ and $E$ such that
\[
  f(x) \leq A\,g(Bx+C)+Dx+E \quad \text{for all $x$.}
\]
Functions $f$ and $g$ are \emph{equivalent},
denoted $f\sim g$, if $f\preceq g$ and $g\preceq f$.
\end{defn}

\begin{rem}
The relation $\preceq$ is an equivalence relation on the set $\mathcal{F}$. Let $f$ and $g$ be two polynomial functions with degree at least $1$ in $\mathcal{F}$ then it is not hard to show that they are equivalent if and only if they have the same degree. Moreover, all exponential functions of the form $a^{bx+c}$, where $a>1$, $b>0$ are equivalent.
\end{rem}

\begin{defn}[Subgroup distortion]
Let $H\leq{G}$ be a pair of finitely generated groups, and let $\mathcal{S}$ and $\mathcal{A}$ be finite generating sets of $G$ and $H$ respectively. The \emph{distortion} of $H$ in $G$ is the function
\[
   \Delta_{H}^G(n)
     = \max \bigset{\abs{h}_{\mathcal{A}}}{\text{$h\in{H}$ and $\abs{h}_{\mathcal{S}}\leq{n}$} }
\]
Up to equivalence, the function $\Delta_{H}^{G}$ does not depend on the choice of finite generating sets $\mathcal{S}$ and $\mathcal{A}$.
\end{defn}

It is well known that a group acting properly, cocompactly, and isometrically on a geodesic space is quasi-isometric to the space.
The following corollary of this fact allows us to compute distortion using the geometries of spaces in place of word metrics.

\begin{cor}
\label{cor:GeometricDistortion}
Let $X$ and $Y$ be compact geodesic spaces, and let $g\colon{(Y,y_0)} \to (X,x_0)$ be $\pi_1$--injective. We lift the metrics on $X$ and $Y$ to geodesic metrics on the universal covers $\tilde{X}$ and $\tilde{Y}$ respectively. Let $G=\pi_1(X,x_0)$ and $H=g_{*} \bigl( \pi_1(Y,y_{0}) \bigr)$.
Then the distortion $\Delta^G_H$ is equivalent to the function
\[
   f(n) = \max \bigset{d_{\tilde Y}(\tilde{y}_0,h (\tilde{y}_0))}{\text{$h \in H$ and $d_{\tilde X}(\tilde{x}_0,h (\tilde{x}_0)) \le n$}}.
\]
\end{cor}

The following propositions is routine, and we leave the proof as an exercise for the reader.

\begin{prop}
\label{prop:distortion}
Let $K',K$ and $G'$ be finitely generated subgroups of a finitely generated group $G$ such that $K' \le G'$ and $K' \le K$. Suppose that $K'$ is undistorted in $K$ and $G'$ is undistorted in $G$ Then $\Delta_{K'}^{G'} \preceq \Delta_{K}^{G}$.
\end{prop}

\begin{prop}
\label{prop:distortion2}
Let $G$, $H$, $K$ be finitely generated groups with $K \le H \le G$.
\begin{enumerate}
   \item If $H$ is a finite index subgroup of $G$ then $\Delta_{K}^{H} \sim \Delta_{K}^{G}$.
   \item If $K$ is a finite index subgroup of $H$ then $\Delta_{K}^{G} \sim \Delta_{H}^{G}$.   \qed
\end{enumerate}
\end{prop}

\begin{lem}[Proposition~9.4 \cite{Hruska2010}]
\label{lem:HruskaProp9.4}
Let $G$ be a finitely generated group with a word length metric $d$. Suppose $H$ and $K$ are subgroups of $G$. For each constant $r$ there is a constant $r' = r'(G,d,H,K)$ so that in the metric space $(G,d)$ we have 
\[
\mathcal{N}_{r}(H) \cap \mathcal{N}_{r}(K) \subset \mathcal{N}_{r'}(H \cap K)
\]
\end{lem}

\begin{defn}
A function $f \colon \N \to \N$ is \emph{superadditive} if 
\[
f(a+b) \ge f(a) + f(b) \, \, \, \text{for all $a,b \in \N$}
\]
The \emph{superadditive closure} of a function $f \colon \N \to \N$ is the function defined by the formula
\[
\bar{f}(n) = \max \bigset{f(n_1)+ \cdots + f(n_{\ell})}{\ell \ge 1 \text{\,\,and $n_1 +\cdots +n_{\ell} = n$}}
\]
\end{defn}
\begin{rem}
\label{rem:superadditive}
The following facts are easy to verify. We leave it as an exercise to the reader.
\begin{enumerate}
    \item Suppose that $f_{i} \sim g_{i}$ with $i =1,\dots,\ell$. Let $f(n) = \max \bigset{f_{i}(n)}{i=1,\dots,\ell}$ and $g(n) = \max \bigset{g_{i}(n)}{i=1,\dots,\ell}$. Then $f \sim g$.
    \item If $f$ and $g$ are superadditive and $f\sim g$ then $\bar{f} \sim \bar{g}$.
\end{enumerate}
\end{rem}

\section{Surfaces in non-geometric 3-manifolds}
\label{sec:manifolds}
In this section, we review backgrounds of surfaces in $3$-manifolds. Throughout this paper, a $3$--manifold is alway assumed to be compact, connected, orientable, irreducible with empty or toroidal boundary. A surface is always compact, connected and orientable and not a $2$--sphere $S^2$.

\begin{defn}
Let $M$ be a compact, orientable, irreducible $3$--manifold with empty or toroidal boundary. The $3$--manifold $M$ is \emph{geometric} if its interior admits a geometric structure in the sense of Thurston which are $3$--sphere, Euclidean $3$--space, hyperbolic $3$-space, $S^{2} \times \R$, $\mathbb{H}^{2} \times \R$, $\widetilde{SL}(2,\R)$, Nil and Sol. Otherwise, $M$ is called \emph{non-geometric}. By the Geometrization Theorem, a non-geometric $3$--manifold can be cut into hyperbolic and Seifert fibered ``blocks'' along a JSJ decomposition. It is called a \emph{graph manifold} if all the blocks are Seifert fibered, otherwise it is a \emph{mixed manifold}. A Seifert fibered space is called \emph{nonelementary} if it is a circle bundle over a hyperbolic $2$--orbifold. 
\end{defn}

A non-geometric $3$--manifold $M$ always has a double cover in which all Seifert fibered blocks are nonelementary. In this section, we will always assume that all Seifert fibered blocks in a non-geometric $3$--manifold are nonelementary.

\begin{defn}
Let $M$ be a Seifert fibered space, and $S \looparrowright M$ is a properly immersed $\pi_1$--injective surface. The surface $S$ is called \emph{horizontal} if it intersects transversely to the Seifert fibers, \emph{vertical} if it is a union of the Seifert fibers.
\end{defn}

The definition of geometrically finite surface below is one of many equivalent forms. We refer the reader to \cite{Bowditch93} for detail discussions.
\begin{defn}
Let $g \colon S \looparrowright M$ be a properly immersed $\pi_1$--injective surface in a hyperbolic $3$--manifold $M$. The surface $S$ is called \emph{geometrically finite} if $\pi_1(S)$ is undistorted subgroup of $\pi_1(M)$, \emph{geometrically infinite} if $S$ is not a geometrically finite surface. 
\end{defn}

\begin{defn}
A properly immersed $\pi_1$--injective surface $g \colon S \looparrowright M$ is called \emph{virtual fiber} if after applying a homotopy relative to boundary, $g$ can be lifted to some finite cover $M_{S}$ of $M$ that fibers over the circle such that $g$ lifts to a fiber. In fact, $M_{S}$ is the mapping torus 
\[
M_{S} = \frac{S \times [0,1]}{(x,0) \sim (\phi(x),1)}
\] for some homeomorphism $\phi$ of $S$.
\end{defn}

\begin{rem}
\label{rem:virtual fiber}
Horizontal surfaces in Seifert fibered spaces and geometrically infinite surfaces in hyperbolic manifolds are all virtual fiber. In particular, if $g \colon S \looparrowright M$ is a horizontal surface in a nonelementary Seifert fibered space $M$ then we may choose $\phi$ as the identity map of $S$ 
(see Lemma~2.1 \cite{RW98}). By Subgroup Tameness Theorem (a combination of Tameness Theorem \cite{Agol04}, \cite{CalegariGabai06} and Canary's Covering Theorem \cite{Canary96}), if $g \colon S \looparrowright M$ is geometrically infinite surface in a hyperbolic manifold $M$ then we may choose $\phi$ as a pseudo-Anosov homeomorphism of $S$ stabilizing each component of $\partial S$, fixing periodic points on $\partial S$. 
In addition, the finite cover map $M_{S} \to M$ takes $S \times \{0\}$ to the image $g(S)$, and $g$ lifts to an embedding $g' \colon S \hookrightarrow M_{S}$ (up to homotopy) where $g'(S)$ is the surface fiber $S \times \{0\}$ in $M_{S}$.
\end{rem}

\begin{defn}
\label{defn:essential}
A properly immersed surface $g \colon (B,\partial B) \looparrowright (M,\partial M) $ is called \emph{essential} if it is not homotopic (relative to $\partial B$) to a map $B \to \partial M$ and the induced homomorphism $g_{*} \colon \pi_1(B) \to \pi_1(M)$ is injective. A loop in the surface $S$ is an \emph{essential curve} if it is neither nullhomotopic or homotopic into the boundary of $S$.
\end{defn}

\begin{rem}
\label{rem:distortionwellknow}
The distortion of a horizontal surface subgroup in a Seifert fibered space group is linear (see \cite{Hruska-Nguyen}) and the distortion of a geometrically infinite surface subgroup in a hyperbolic manifold group is exponential (by Subgroup Tameness Theorem). 
\end{rem}

\begin{defn}[Clean Surface]
\label{defn:clean surface}
Let $N$ be a non-geometric $3$--manifold, and $\mathcal{T}$ the union of JSJ tori. Let $S$ be a compact, orientable, connected surface. Let $g \colon S \looparrowright N$ be an immersion such that $S$ and $\mathcal{T}$ intersects transversely. The immersion is called \emph{clean surface} in $N$ if the following holds.
\begin{enumerate}
    \item $g(\partial{S}) \subset \mathcal{T} \cup \partial{N}$
    \item $g(S-\partial S) \cap \partial N = \emptyset$
    \item $S$ intersects the JSJ tori of $N$ in a minimal finite collection $\mathcal{T}_{g}$ of disjoint essential curves of $S$.
    \item The complementary components of the union of curves in $\mathcal{T}_{g}$ are essential subsurfaces (in the sense of Definition~\ref{defn:essential}) of $S$, called \emph{pieces} of $S$. Each piece of $S$ is mapped into either a hyperbolic block or Seifert fibered block of $N$. Each piece of $S$ in a hyberbolic block is either geometrically finite or geometrically infinite. Each piece of $S$ in a Seifert fibered block is either horizontal or vertical.
    \item Let $N' \to N$ be the covering space corresponding to the subgroup $\pi_1(S)$ of $\pi_1(N)$. The immersion $g$ lifts to an embedding $S \to N'$.
\end{enumerate}
\end{defn}

\begin{defn}
The \emph{almost fiber part} $\Phi(S)$ of $S$ is the union of all the horizontal or geometrically infinite pieces mapped into Seifert fibered or hyperbolic blocks of $N$ respectively. The surface $S$ is called \emph{almost fiber} if $\Phi(S) =S$. 
\end{defn}

\begin{rem}
\begin{enumerate}
    \item Any properly immersed $\pi_1$--injective surface $g \colon S \looparrowright N$ with $S$ compact, orientable, connected and not homeomorphic to $S^{2}$ is homotopic to a clean surface. 
    \item Each component of the almost fiber part of a clean surface is a clean almost fiber surface.
\end{enumerate}
\end{rem}

Rubinstein-Wang \cite{RW98} introduces a combinatorial invariant to characterize the virtual embedding of a horizontal surface $S$ in graph manifold $N$ (i.e, after applying a homotopy, the immersion lifts to an embedding of $S$ in some finite cover of $N$). In \cite{YiLiu2017}, Liu generalizes the invariant of Rubinstein-Wang, which he calls \emph{spirality} (this concept has also been called ``dilation'' in \cite{Hruska-Nguyen}), to surfaces in closed $3$--manifold $N$, and proves that spirality is the obstruction to the surface being virtually embedded. Recently, Sun \cite{Sun18} generalizes Liu's work to separability of arbitrary finitely generated subgroup in non-geometric $3$--manifolds.

There are two equivalent definitions of spirality given by \cite{YiLiu2017}. Liu first defines spirality by partial dilations and a principal $\Q^{\times}$--bundle over $\Phi(S)$. Liu then gives a combinatorial formula (Formula~4.5 in Secttion~4.2 \cite{YiLiu2017}) and shows that spirality can be computed by this formula. The definition of spirality below is from Section~4.2 in \cite{YiLiu2017} that also can be seen in Section~3.3 \cite{Sun18}.

\begin{defn}[Spirality]
\label{defn:DilationSlopes}
Let $g \colon S \looparrowright N$ be a clean surface in a non-geometric $3$--manifold $N$. With respect to $\mathcal{T}_{g}$, let $\Gamma(\Phi(\mathcal{T}_{g}))$ be the dual graph of $\Phi(S)$. 
For each vertex $v$ of $\Gamma(\Phi(\mathcal{T}_{g}))$, let $B_{v}$ be the piece of $S$ corresponding to the vertex $v$, and let $M_{v}$ be the block of $N$ such that $B_{v}$ is mapped into $M_{v}$. We choose a mapping torus \[
M_{B_v} = \frac{B_{v} \times [0,1]}{(x,0) \sim (\phi_{v}(x),1)}
\] as in Remark~\ref{rem:virtual fiber}. 
For each directed edge $e$ in $\Gamma(\Phi(\mathcal{T}_{g}))$ with $v$ as its initial vertex. Let $c_{e}$ be the circle boundary of $B_{v}$ corresponding to $e$. Let $T_{e}$ be the boundary torus of $M_{v}$ containing $c_{e}$. Let $T'_{e}$ be the boundary torus of $M_{B_v}$ containing $c_{e}$. We associate to $c_e$ a nonzero integer 
$h_e = [T_{e}': T_e]$
where $[-:-]$ denotes the covering degree. Let $-e$ denote $e$ with the orientation reversed. Let 
\[
\xi_{e} = h_{e}\bigl /h_{-e}
\]
There is a natural homomorphism $w \colon H_{1}(\Phi(S);\Z) \to \Q^{\times}$ defined as follows.
For any directed $1$--cycle $\gamma$ in $\Phi(S)$ dual to a cycle of directed edges $e_1, \dots ,e_{n}$ in $\Gamma(\Phi(\mathcal{T}_{g}))$, the \emph{spirality} of $\gamma$ is the number 
\[
w(\gamma) = \prod_{i=1}^{n}\xi_{e_i}
\]
We say the \emph{spirality of $S$ is trivial} if $w$ is a trivial homomorphism. The \emph{governor} of $g$ with respect to the chosen mapping torus $M_{B_{v}}$ is the maximum of values $\xi_{e}$ with $e$ varying over all directed edges in the graph $\Gamma(\Phi(\mathcal{T}_{g}))$.
\end{defn}

\begin{rem}
\label{rem:SUNLIU}
\begin{enumerate}
    \item It is shown by Yi Liu in \cite{YiLiu2017} that the homomorphism $w$ does not depend on the choice of mapping torus $M_{B_v}$. Moreover, Yi Liu shows that if $N$ is a closed manifold, and $S$ is a closed surface then $\pi_1(S)$ is separable in $\pi_1(N)$ if and only if the spirality of $S$ is trivial (see Theorem~1.1 \cite{YiLiu2017}). Recent work of Sun (see Theorem~1.3 in \cite{Sun18}) allows us to say that fundamental group of a clean surface $S$ in a non-geometric $3$--manifold $N$ is separable if and only if the spirality of $S$ is trivial.
    \item When $N$ is a graph manifold and $S$ is horizontal, properly immersed then the notion of spirality in Definition~\ref{defn:DilationSlopes} was previously studied by Rubinstein-Wang \cite{RW98}.
\end{enumerate}
\end{rem}

The proof of the following proposition is essentially the same as Proposition~4.15 in \cite{Hruska-Nguyen}.
\begin{prop}
\label{prop:virtualupper}
For each $\gamma \subset \Phi(S)$ as in Definition~\ref{defn:DilationSlopes}, we define 
\[
\Lambda_{\gamma} =\max \bigset{\prod_{i=j}^{k} \xi_{e_i}}{1 \le j \le k \le n}
\]
If the spirality of $S$ is trivial, then there exists a positive constant $\Lambda$ such that $\Lambda_{\gamma} \le \Lambda$ for all all directed $1$--cycle $\gamma$ in $\Phi(S)$.
\end{prop}

\begin{defn}
\label{defn:flowline}
Let $F$ be a compact, orientable connected surface with non-empty boundary and $\chi(F) < 0$. Let $\varphi \colon F \to F$ be a orientation preserving homeomorphism fixing $\partial{F}$ setwise. Let $M_{F} = F \times [0,1] \bigl / (x,0) \sim (\varphi(x),1)$. 
Projection of $F \times [0,1]$ onto the second factor induces a map $\sigma: M_{F} \to S^1$ which is a fibration with fiber $F$. The foliation of $F \times [0,1]$ by intervals has image in $M_{F}$ a one-dimensional foliation which we denote by $\mathcal{L}$ called the \emph{suspension flow} on $M_{F}$.
\end{defn}

In the following, when we say $\varphi$ \emph{fixes periodic points} on $\partial F$, we mean all the periodic points of $\varphi$ on $\partial F$ are fixed points of $\varphi$. We note that if $\varphi$ is a pseudo-Anosov, then after passing to a power $\varphi^{m}$ of $\varphi$ for some sufficiently large integer $m$, the map $\varphi^{m}$ fixes all periodic points of $\varphi$ on $\partial F$.

\begin{defn}[Degeneracy slope]
\label{defn:degeneracyslope}
If the map $\varphi$ in Definition~\ref{defn:flowline} fixes periodic points on $\partial F$ then on each boundary component of $M_{F}$, there exists a closed leaf (of the suspension flow), and different closed leaves in the same boundary component are parallel to each other. We will call any such leaf a \emph{degeneracy slope}.
Each boundary component $c$ of $F$ is mapped into a boundary torus of $M_{F}$, we fix a degeneracy slope on this torus, and denoted it by $\mathbf{s}_{cF}$.
\end{defn}
Let $f \colon F \times \R \to F \times \R$ be the homeomorphism given by $f(x,t) = \bigl (\varphi(x),t+1 \bigr )$. We denote $\langle f \rangle$ be the infinite cyclic group generated by $f$ and $\hat{M}_{F} = F \times \R$. We note that the quotient space $F \times \R \bigl / \langle f \rangle $ is the mapping torus $M_{F}$. Let the triple $\bigl (\hat{M}_{F},\theta^{1},\theta^{2} \bigr )$ be the pullback bundle of the fibration $\sigma \colon M_{F} \to S^1$ by the infinite cyclic covering map $\R \to S^1$ where $\theta^{2} \colon F \times \R \to \R$ is the projection on the second factor and $\theta^{1}$ is the quotient map $F \times \R \to F \times \R \bigl / \langle f \rangle $. 
The universal cover $\tilde{M}_{F}$ is identified with $\tilde{F} \times \R$. For each integer $n$, the subspace $\tilde{F} \times \{n\}$ of $\tilde{M}_{F} = \tilde{F} \times \R$ is called a \emph{slice} of $\tilde{M}$.
We have the following lemma.

\begin{lem}
\label{lem:height}
Let $M_{F}$ be the mapping torus of a orientation preserving homeomorphism $\varphi$ of a compact orientable connected surface $F$ with nonempty boundary and $\chi(F) <0$. We assume that $\varphi$ fixes $\partial{F}$ setwise and $\varphi$ fixes periodic points on $\partial{F}$. Equip $M_{F}$ with a length metric, and let $d$ be the metric on $\tilde{M}_{F}$ induced from the metric on $M_{F}$. There are positive constants $L$ and $C$ such that for any $x$ in the slice $\tilde{F} \times \{n\}$ and $y$ in the slice  $\tilde{F} \times \{m\}$ then 
\[
\abs{m-n} \le L\,d(x,y) + C
\]
\end{lem}
\begin{proof}
Let $x_{0}$ be a point on a boundary circle of $\partial{F}$ and $\tilde{x}_{0}$ be a lift of $x_{0}$ in $\tilde{M}_{F}$. We fix generating sets $\mathcal{A}$ and $\mathcal{B}$ of $\pi_1(F,x_0)$ and $\pi_1(M_{F},x_0)$ respectively.
We remark that there is a positive constant $\epsilon >0$ such that for any integer $k$ and for any $z$ in the slice $\tilde{F} \times \{k\}$ of $\tilde{M}_{F}$, there exists $z'$ in the slice $\tilde{F} \times \{k\}$ such that $z'$ is a lift of the base point $x_0$ and $d(z,z') \le \epsilon$. 
Choose $x'$ in the slice $\tilde{F} \times \{n\}$ and $y'$ in the slice $\tilde{F} \times \{m\}$ so that $x'$ and $y'$ are lifts of $x_0$ with $d(x,x') \le \epsilon$ and $d(y,y') \le \epsilon$.

Let $\sigma \colon M_{F} \to S^1$ be the projection of the bundle $M_{F}$. It follows that we have the short exact sequence:
\[
1 \to \pi_1(F,x_0) \to \pi_1(M_{F},x_0) \to \Z \to 1
\]
Since $\sigma_{*}$ is a homomorphism, it is easy to see that there exists $L' >0$ such that 
\[
\abs{\sigma_{*}(g) - \sigma_{*}(g')} \le L'\,\abs{g-g'}_{\mathcal{B}}
\] for all $g, g' \in \pi_1(M_{F},x_0)$. 

Since $\pi_1(M_{F},x_0)$ acts geometrically on $\tilde{M}_{F}$, it follows that there exist constants $A \ge 1$ and $B \ge 0$ such that 
\[
\abs{g-g'}_{\mathcal{B}} \le A \,d(g(\tilde{x}_0),g'(\tilde{x}_0)) + B
\] for all $g, g' \in \pi_1(M_{F},x_0)$.
It follows that $\abs{m-n} \le L'\,A \,d(x',y') + L'\,B$ since $x'$ and $y'$ are lifts of $x_0$. Since $d(x',y') \le d(x,y) + 2\epsilon$, it follows that 
\[
\abs{m-n} \le Ld(x,y) + C
\] where $L= L'\,A$ and $C = L'\,B + 2L'\,A\epsilon$.
\end{proof}

The following lemma can be seen in the proof of Theorem~11.9 in \cite{FLP}.
\begin{lem}
\label{lem:pseudo}
Let $B$ be a surface with nonempty boundary with $\chi(B) <0$. Let $\varphi \colon B \to B$ be a pseudo-Anosov homeomorphism fixing the boundary $\partial{B}$ setwise. Let $\alpha$ be a geodesic such that $\alpha(0)$ and $\alpha(1)$ belong to a boundary circle of $B$ and $\alpha$ is not homotoped to a boundary circle. For any $n \in \N$, let $\gamma_{n}$ be the geodesic connecting the two endpoints of a lift of $\varphi^{n}(\alpha)$ in the universal cover $\tilde{B}$, and let $\beta_{n}$ be the shortest path in $\tilde{B}$ joining two boundary lines containing the endpoints of $\gamma_{n}$. Then 
\begin{enumerate}
   \item $\limsup\limits_{n\rightarrow \infty}\ln{\bigl (\bigabs{\gamma_{n}}_{\tilde{B}} \bigr )} \bigl / n = \lambda >1$
   \item 
   \label{item:superexponentialgrowth}
   $\limsup\limits_{n\rightarrow \infty}\ln d \bigl (\beta_{n}(0),\gamma_{n}(0) \bigr ) \bigl / n = 0$ and $\limsup\limits_{n\rightarrow \infty}\ln d \bigl (\beta_{n}(1),\gamma_{n}(1) \bigr ) \bigl / n = 0$. 
\end{enumerate}
\end{lem}

\subsection{Metrics on non-geometric 3-manifolds}
\label{rem:metriconmixedmanifold}
Since we compute the distortion of a surface subgroup in non-geometric $3$--manifold group by using geometry of their universal covers (see Corollary~\ref{cor:GeometricDistortion}), we need to discuss the metrics on non-geometric $3$--manifolds that we are going to use. We note that the choice of length metrics does not affect the distortion, so we will choose a convenient metric. 

{\bf Metrics on mixed $3$--manifolds:}
In the rest of this paper, if we are working on the setting of mixed manifolds, the following metric is the metric we will talk about. If $N$ is a mixed manifold, it is shown by Leeb \cite{Leeb95} that $N$ admits a smooth Riemannian metric $d$ of nonpositive sectional curvature with totally geodesic boundary such that $\mathcal{T}$ is totally geodesic and the sectional curvature is strictly negative on each hyperbolic component of $N\setminus \mathcal{T}$. 

{\bf Metrics on simple graph manifolds:}  A \emph{simple graph manifold} $N$ is a graph manifold with the following properties: Each Seifert component is a trivial circle bundle over an orientable surface of genus at least 2. The intersection numbers of fibers of adjacent Seifert components have absolute value 1. It was shown by Kapovich and Leeb that any graph manifold $N$ has a finite cover $\hat{N}$ that is a simple graph manifold \cite{KL98}.

In the rest of this paper, if we are working on the setting of simple graph manifolds, the following metric (described by Kapovich--Leeb \cite{KL98}) will be the metric we will talk about. If $N$ is a simple graph manifold, on each Seifert fibered block $M_{i} = F_{i} \times S^1$ we choose a hyperbolic metric on $F_{i}$ and then equip $M_{i}$ with the product metric $d_{i}$. There is a length metric $d$ on $N$ with the following properties. There is $K>0$ such that for each Seifert fibered block $M_{i}$, we have
\[
\frac{1}{K}d_{i}(x,y) \le d(x,y) \le Kd_{i}(x,y)
\] for all $x$ and $y$ in $M_{i}$.

\begin{rem}
\label{rem:constant rho}
There exists a positive lower bound $\rho$ for the distance between any two distinct JSJ planes in $\tilde{N}$.
\end{rem}

\section{Distortion of surfaces is determined by the almost fiber part}
\label{sec:almostfiber}
The goal in this section is to show that the distortion of the fundamental group of a surface $S$ in the fundamental group of a non-geometric $3$--manifold $N$ can be determined by looking at the distortion of the almost fiber part $\Phi(S)$.

\begin{thm}
\label{thm:distortionalmostfiber}
Let $g \colon S  \looparrowright N$ be a clean surface in a non-geometric $3$--manifold $N$. 
For each component $S_{i}$ of $\Phi(S)$, let $\delta_{S_{i}}$ be the distortion of $\pi_1(S_{i})$ in $\pi_1(N)$. Then the distortion of $H = \pi_1(S)$ in $G= \pi_1(N)$ satisfies
\[
f \preceq \Delta_{H}^{G} \preceq \bar{f}
\]
where
\[
f(n) := \max \bigset{\delta_{S_{i}}(n)}{\text {$S_{i}$ is a component of $\Phi(S)$}}
\] and $\bar{f}$ is the superadditive closure of $f$.
\end{thm}

\begin{rem}
The definition of $f$ depends on choices of generating sets for $\pi_1(N)$ and each $\pi_1(S_{i})$. In general it is unknown whether $f \sim \bar{f}$ for an arbitrary distortion function $f$. But in Section~\ref{sec:distortion mixed} we will see this is true because each function $\delta_{S_{i}}$ is either linear, quadratic, exponential or double exponential.
\end{rem}
We use the convention that $f(n)=0$ if $\Phi(S) = \emptyset$. Note that the zero function is equivalent to a linear function by Definition~\ref{def:equivalentfunction}. Therefore we obtain the following corollary.

\begin{cor}
\label{cor:linear distortion}
Let $g \colon S \looparrowright N$ be a clean surface in a non-geometric $3$--manifold $N$. 
If the almost fiber part $\Phi(S)$ is empty then the distortion of $\pi_1(S)$ in $\pi_1(N)$ is linear.
\end{cor}

Regarding Theorem~\ref{thm:distortionalmostfiber}, the proof $f \preceq \Delta_{H}^{G}$ is not hard to see, meanwhile the proof $\Delta_{H}^{G} \preceq \bar{f}$ requires more work. We sketch here the idea of the proof of the upper bound case. We fix a lifted point $\tilde{s}_{0}$ in $\tilde{S}$, and let $h \in \pi_1(S,s_0)$ such that the distance of $\tilde{s}_{0}$ and $h(\tilde{s}_{0})$ in $\tilde{N}$ is less than $n$. We will construct a path $\gamma'$ in $\tilde{N}$ connecting $\tilde{s}_{0}$ to $h(\tilde{s}_{0})$ such that $\abs{\gamma'}$ is bounded above by a linear function in term of $n$. We then construct a path $\beta$ in $\tilde{S}$ connecting $\tilde{s}_{0}$ to $h(\tilde{s}_{0})$ such that $\beta$ stays close to $\gamma'$ every time they travel in the same block containing a piece which is either vertical or geometrically finite (see Lemma~\ref{lem:verticalgraphmanifold} and Lemma~\ref{lem:geometricallyfinitehyper}).

\begin{lem}
\label{lem:verticalgraphmanifold}
Let $F$ be a connected compact hyperbolic surface with non-empty boundary. Let $M = F \times S^1$.
Let $g \colon (S,s_0) \looparrowright (M,x_0)$ be an essential, vertical surface. We equip $M$ with a length metric and lift this metric to the metric $d$ in the universal covers $\tilde{M}$. Then there exists a constant $R$ such that the following holds.
Let $P$ and $P'$ be two distinct boundary planes in $\tilde{M}$ such that $P \cap \tilde{S} \neq \emptyset$ and $P' \cap \tilde{S} \neq \emptyset$. Let $x$ and $y$ be points in $P$ and $P'$, and $\alpha$ be a geodesic in $(\tilde{M},d)$ connecting $x$ to $y$. Then there exists a path $\beta$ in $\tilde{S}$ connecting a point in $P \cap \tilde{S}$ to a point in $P' \cap \tilde{S}$ such that $\beta(0), \beta(1) \in \mathcal{N}_{R}(\alpha)$.
\end{lem}
\begin{proof}
Since $S$ is orientable and vertical, it follows that $S$ is an annulus. The map $g$ is a vertical map and thus the image $g(S)$ in $M$ is $\gamma \times S^1$ where $\gamma$ is a proper arc in the base surface $F$ of $M$ (i.e, $\gamma$ could not be homotoped to a path in a boundary circle).
We fix a hyperbolic metric $d_{F}$ on $F$ such that the boundary is totally geodesic. We lift the metric $d_{F}$ to the metric $d_{\tilde{F}}$ in the universal cover $\tilde{F}$ of $F$. We equip $\tilde{M} = \tilde{F} \times \R$ with the product metric $d'$. We note that the identity map $(\tilde{M},d) \to (\tilde{M},d')$ is a  $(K,C)$--quasi-isometric for some constant $K$ and $C$. In $\tilde{M}$, we note that $\tilde{S}$ is $\tilde{\gamma} \times \R$ where $\tilde{\gamma}$ is a path lift of $\gamma$ in $\tilde{F}$. 

We note that $(\tilde{F}, d_{\tilde{F}})$ is bilipschitz homeomorphic to a fattened
tree (see the paragraph after Lemma~1.1 \cite{NeuBeh08}). Thus, there exists $A_{0} >0$ such that the following holds. Let $\ell$ and $\ell'$ be two distinct boundary lines in $\tilde{F}$. Let $[p,p']$ be a geodesic of shortest length from $\ell$ to $\ell'$. If $\tau$ is a path in $\tilde{F}$ connecting a point in $\ell$ to a point in $\ell'$ then $[p,p'] \subset \mathcal{N}'_{A_0}(\tau)$ where $\mathcal{N}'_{A_0}(\tau)$ is the $A_0$--neighborhood of $\tau$ with respect to $d_{\tilde{F}}$--metric.

Let $\abs{\tilde{\gamma}}_{\tilde{F}}$ be the length of $\tilde{\gamma}$ with respect to $d_{\tilde{F}}$--metric. Let $L = 2\abs{\tilde{\gamma}}_{\tilde{F}} + 2A_0$. Let $\alpha_{\tilde{F}}$ be the projection of $\alpha$ on the first factor $\tilde{F}$ of $\tilde{M}$. We first show that $\tilde{\gamma} \subset \mathcal{N}'_{L}(\alpha_{\tilde{F}})$. Indeed, let $\ell_1$ and $\ell_2$ be the boundary lines in $\tilde{F}$ such that $\tilde{\gamma}(0) \in \ell_1$ and $\tilde{\gamma}(1) \in \ell_2$. Let $[p_1,p_2]$ be a geodesic of shortest length from $\ell_1$ to $\ell_2$. According to the previous paragraph, we have $[p_1,p_2] \subset \mathcal{N}'_{A_0}(\alpha_{\tilde{F}})$ and $[p_1,p_2] \subset \mathcal{N}'_{A_0}(\tilde{\gamma})$. Since $p_1 \in \mathcal{N}'_{A_0}(\tilde{\gamma})$, it follows that there exists $a \in \tilde{\gamma}$ such that $d_{\tilde{F}}(p_1,a) \le A_0$. Thus $d_{\tilde{F}}(\tilde{\gamma}(0),p_1) \le d_{\tilde{F}}(\tilde{\gamma}(0), a) + d_{\tilde{F}}(a,p_1) \le \abs{\tilde{\gamma}}_{\tilde{F}} + d_{\tilde{F}}(a,p_1) \le \abs{\tilde{\gamma}}_{\tilde{F}} +A_0$. For any $x \in \tilde{\gamma}$, we have $d_{\tilde{F}}(x,p_1) \le d_{\tilde{F}}(x,\tilde{\gamma}(0)) + d_{\tilde{F}}(\tilde{\gamma}(0), p_1) \le \abs{\tilde{\gamma}}_{\tilde{F}} + d_{\tilde{F}}(\tilde{\gamma}(0), p_1) \le \abs{\tilde{\gamma}}_{\tilde{F}} + \abs{\tilde{\gamma}}_{\tilde{F}} +A_0 =2\abs{\tilde{\gamma}}_{\tilde{F}} +A_0 $. It follows that $\tilde{\gamma} \subset \mathcal{N}'_{2\abs{\tilde{\gamma}}_{\tilde{F}} +A_0}([p_1,p_2])$. Using $\tilde{\gamma} \subset \mathcal{N}'_{2\abs{\tilde{\gamma}}_{\tilde{F}} +A_0}([p_1,p_2])$ and $[p_1,p_2] \subset \mathcal{N}'_{A_0}(\alpha_{\tilde{F}})$ we have $\tilde{\gamma} \subset \mathcal{N}'_{2\abs{\tilde{\gamma}}_{\tilde{F}} +2A_0}(\alpha_{\tilde{F}}) = \mathcal{N}'_{L}(\alpha_{\tilde{F}})$. 

Since $\tilde{\gamma} \subset \mathcal{N}'_{L}(\alpha_{\tilde{F}})$, it follows that there exist $u_0 \in \alpha_{\tilde{F}}$ and $u_1 \in \alpha_{\tilde{F}}$ such that $d_{\tilde{F}}(\tilde{\gamma}(0),u_0) \le L$ and $d_{\tilde{F}}(\tilde{\gamma}(1),u_1) \le L$. Choose $s_{0}, s_1 \in \R$ such that $(u_0, s_0), (u_1,s_1) \in \alpha$. It follows that $d \bigl ((\tilde{\gamma}(0),s_0), (u_0,s_0) \bigr ) \le Kd' \bigl ((\tilde{\gamma}(0),s_0), (u_0,s_0) \bigr ) + C = Kd_{\tilde{F}}(\tilde{\gamma}(0),u_0) + C \le KL +C$. Hence $(\tilde{\gamma}(0),s_0) \in \mathcal{N}_{KL+C}(\alpha)$. Similarly, we have $(\tilde{\gamma}(1),s_1) \in \mathcal{N}_{KL+C}(\alpha)$. Note that $(\tilde{\gamma}(0),s_0)$ and $(\tilde{\gamma}(1),s_1)$ are in $\tilde{S} = \tilde{\gamma} \times \R$. Let $\beta$ be a path in $\tilde{S}$ connecting $(\tilde{\gamma}(0),s_0)$ to $(\tilde{\gamma}(1),s_1)$. The end points of $\beta$ are in $\mathcal{N}_{R}(\alpha)$.
\end{proof}

 \begin{lem}
 \label{lem:geometrically finite parallel lines}
 Let $g \colon (S,s_0) \looparrowright (M,x_0)$ be a essential, geometrically finite surface in a hyperbolic manifold $M$ with nonempty toroidal boundary such that $\partial S \neq \emptyset$. Let $\tilde{g} \colon (\tilde{S},\tilde{s}_{0}) \hookrightarrow (\tilde{M},\tilde{x}_{0})$ be a lift of $g$. Then for any distinct boundary lines $\ell$ and $\ell'$ of $\partial \tilde{S}$, the images $\tilde{g}(\ell)$ and $\tilde{g}(\ell')$ lie in different boundary planes of $\partial \tilde{M}$.
 \end{lem}
 
\begin{proof}
Suppose by the way of contradiction that $\tilde{g}(\ell)$ and $\tilde{g}(\ell')$ are lines in the same boundary plane $\tilde{T}$.
Since $g \colon (S,s_0) \looparrowright (M,x_0)$ is essential, $S$ could not be annulus or a disk. Thus, $S$ is a hyperbolic surface. Let $d_{S}$ be a hyperbolic metric on $S$ such that the boundary is totally geodesic and let $d_{M}$ be a non-positively curved metric on the manifold with boundary $M$. We lift these metrics to metrics $d_{\tilde{S}}$ and $d_{\tilde{M}}$ in the universal covers $\tilde{S}$ and $\tilde{M}$ respectively. Since $g \colon (S,s_0) \looparrowright (M,x_0)$ is geometrically finite, it follows that $\tilde{g} \colon (\tilde{S},d_{\tilde{S}}) \hookrightarrow (\tilde{M},d_{\tilde{M}})$ is an $(L,C)$--quasi-isometric embedding for some constant $L$ and $C$.

Since $\tilde{g}$ is an embedding, it follows that $\tilde{g}(\ell)$ and $\tilde{g}(\ell')$ are disjoint lines in $\tilde{T}$. We note that, on the one hand the Hausdorff distance of two sets $\ell$ and $\ell'$ with respect to $d_{\tilde{S}}$--metric is infinite (this follows from Lemma~3.2 in \cite{Hruska-Nguyen}). On the other hand, the Hausdorff distance of two sets $\tilde{g}(\ell)$ and $\tilde{g}(\ell')$ with respect to $d_{\tilde{M}}$--metric is finite. (this follows from the fact that $A= stab(\tilde{T})$ in $\pi_1(M)$ acts isometrically on $\tilde{T}$ and $stab(\tilde{g}(\ell))$ and $stab(\tilde{g}(\ell'))$ are commensurable in $A$). This could not happen since $\tilde{g}$ is a quasi-isometric embedding.
\end{proof}

\begin{rem}
Lemma~\ref{lem:geometrically finite parallel lines} can be proven by using malnormality of the peripheral subgroups of $\pi_1(S)$.
\end{rem}

\begin{lem}
\label{lem:geometricallyfinitehyper}
Let $M$ be a hyperbolic manifold with nonempty toroidal boundary. Let $g \colon (S,s_0) \looparrowright (M,x_0)$ be a essential, geometrically finite surface such that $\partial S \neq \emptyset$. Equip $M$ with a non-positively curved metric and lift this metric to the universal cover $\tilde{M}$ denoted by $d$.
Then there exists a constant $R$ such that the following holds.
Let $P$ and $P'$ be two distinct boundary planes in $\tilde{M}$ such that $P \cap \tilde{S} \neq \emptyset$ and $P' \cap \tilde{S} \neq \emptyset$. Let $x$ and $y$ be points in $P$ and $P'$ respectively, and $\alpha$ be a geodesic in $\tilde{M}$ connecting $x$ to $y$. Then there is a path $\beta$ in $\tilde{S}$ connecting a point in $P \cap \tilde{S}$ to a point in $P' \cap \tilde{S}$ such that $\beta(0), \beta(1) \in  \mathcal{N}_{R}(\alpha)$.
\end{lem}

\begin{proof}
Let $G = \pi_1(M,x_0)$ and $H = \pi_1(S,s_0)$. Let $\PP$ be the collection of fundamental groups of tori boundary of $M$. Since $g \colon (S,s_0) \looparrowright (M,x_0)$ is geometrically finite, it follows that $\pi_1(S,s_0)$ is relatively quasiconvex in the relatively hyperbolic group $(G,\PP)$ (see Corollary~1.6 in \cite{Hruska2010}).
Since $d$ is a complete non-positively curved metric, it follows from Cartan-Hadamard Theorem that $(\tilde{M},d)$ is a $\CAT(0)$ space. It also follows from Corollary~1.6 in \cite{Hruska2010} that the orbit space $\pi_1(S,s_0)(\tilde{x}_{0})$ is quasiconvex in $(\tilde{M},d)$. It follows that $\tilde{S}$ is $\epsilon_{0}$--quasiconvex in $(\tilde{M},d)$ for some positive constant $\epsilon_0$.

Applying Lemma~\ref{lem:HruskaProp9.4} to the surface subgroup and the fundamental group of each torus boundary, we have the following fact: For any $r>0$, there exists $r' =r'(r)>0$ such that whenever $x \in \mathcal{N}_{r}(\tilde{T}) \cap \mathcal{N}_{r}(\tilde{S})$ and $\tilde{T}$ is an arbitrary boundary plane of $\tilde{M}$ with nonempty intersection with $\tilde{S}$, then $x \in \mathcal{N}_{r'}(\tilde{T} \cap \tilde{S})$.

We note that $(\tilde{M},d)$ is a $\CAT(0)$ space with isolated flats. Let $\epsilon_{1}$ be the positive constant given by Proposition~8 \cite{HK09}. Let $[p,q]$ be a geodesic of shortest length from $P$ to $P'$. Then every geodesic from $P$ to $P'$ must come within a distance $\epsilon_{1}$ of both $p$ and $q$. Since $\alpha$ is a geodesic in $\tilde{M}$ connecting $x \in P$ to $y \in P'$, it follows that $\{p,q\} \in \mathcal{N}_{\epsilon_1}(\alpha)$. Moreover, there exist points $x'$ and $y'$ in a geodesic $\gamma$ from $P \cap \tilde{S}$ to  $P' \cap \tilde{S}$ such that $d(x',p) \le \epsilon_{1}$ and $d(y',q) \le \epsilon_{1}$. Hence $x' \in \mathcal{N}_{\epsilon_1}(P)$ and $y' \in \mathcal{N}_{\epsilon_1}(P')$. We note that the end points of $\gamma$ belong to $\tilde{S}$. Using quasiconvexity of $\tilde{S}$, we have $x',y' \in \mathcal{N}_{\epsilon_0}(\tilde{S})$. Thus there exists a constant $\epsilon_{2}$ depending on $\epsilon_{0}$ and $\epsilon_{1}$ such that $x' \in \mathcal{N}_{\epsilon_2}(P) \cap \mathcal{N}_{\epsilon_2}(\tilde{S})$ and $y' \in \mathcal{N}_{\epsilon_2}(P') \cap \mathcal{N}_{\epsilon_2}(\tilde{S})$ (we may choose $\epsilon_2 = \epsilon_0 + \epsilon_1$).
Let $r' = r'(\epsilon_2)$ be the constant given in the previous paragraph with respect to $\epsilon_2$. It follows that $x' \in \mathcal{N}_{r'}(P \cap \tilde{S})$ and $y' \in \mathcal{N}_{r'}(P' \cap \tilde{S})$. Thus, $d(x',u) \le r'$ and $d(y',v) \le r'$ for some points: $u \in P \cap \tilde{S}$ and $v \in P' \cap \tilde{S}$. Let $\beta$ be a path in $\tilde{S}$ connecting $u$ to $v$. Since $d(\beta(0),p) = d(u,p) \le d(u,x') +d(x',p) \le r'+\epsilon_1$ and $p \in \mathcal{N}_{\epsilon_1}(\alpha)$, it follows that $\beta(0) \in \mathcal{N}_{r'+2\epsilon_1}(\alpha)$. Similarly, since  $d(\beta(1),q) = d(v,q) \le d(v,y') +d(y',q) \le r'+\epsilon_1$ and $q \in \mathcal{N}_{\epsilon_1}(\alpha)$, it follows that $\beta(1) \in \mathcal{N}_{r'+2\epsilon_1}(\alpha)$. Let $R =r'+2\epsilon_1$, the lemma is confirmed.

\end{proof}

Let $g \colon S \looparrowright N$ be the immersion in the statement of Theorem~\ref{thm:distortionalmostfiber}.

\begin{defn}
\label{defn:map of trees}
Lift the JSJ decomposition of the manifold $N$ to the universal cover $\tilde N$, and let $\mathbf T_N$ be the tree dual to this decomposition of $\tilde N$.
Lift the collection $\mathcal T_g$ to the universal cover $\tilde S$.  The tree dual to this decomposition of $\tilde S$ will be denoted by $\mathbf T_S$. The map $\tilde g$ induces a map $\zeta \colon{\mathbf{T}_S} \to \mathbf{T}_N$.
\end{defn}

\begin{rem}
\label{rem:flowlineintersectone}
For each geometrically infinite piece $\tilde{B}$ in $\tilde{S}$, let $\tilde{M}$ be the block of $\tilde{N}$ such that $\tilde{g}(\tilde{B}) \subset \tilde{M}$. By Remark~\ref{rem:virtual fiber}, the immersion $B \looparrowright M$ lifts to an embedding (up to homotopy) in a finite cover $M_B$ of $M$ which is fibered over circle with a fiber $B$. Let $\mathcal{L}_{M_B}$ be the suspension flow on $M_B$. We note that $\tilde{B}$ meets every flow line of $\tilde{\mathcal{L}}_{M_B}$ once. 
\end{rem}

\begin{prop}
\label{prop:TreeBijective}
The map $\zeta$ is injective. 
\end{prop}

\begin{proof}
A simplicial map between trees is injective if it is locally injective (see \cite{Stallings83}). Suppose by way of contradiction that $\zeta$ is not locally injective. Then there exists three distinct pieces $\tilde{B}_1$, $\tilde{B}_2$ and $\tilde{B}_3$ in $\tilde{S}$ such that $\tilde{B}_1 \cap \tilde{B}_2$ is a line $\ell_{1}$ and $\tilde{B}_2 \cap \tilde{B}_3$ is a line $\ell_{2}$ and the images $\tilde{g}(\tilde{B}_{1})$ and $\tilde{g}(\tilde{B}_{3})$ lie in the same block $\tilde{M}_{1}$ of $\tilde{N}$. Let $\tilde{M}_{2}$ be the block containing the image $\tilde{g}(\tilde{B}_{2})$. We have $\tilde{g}(\ell_1)$ and $\tilde{g}(\ell_2)$ are subsets of the JSJ plane $\tilde{T} = \tilde{M}_1 \cap \tilde{M}_2$.
Since the map $\tilde{g}$ is an embedding, it follows that $\tilde{g}(\ell_1)$ and $\tilde{g}(\ell_2)$ are disjoint lines in the plane $\tilde{T}$.
If $\tilde{B}_{2}$ is horizontal, this contradicts Lemma~6.3 in \cite{Hruska-Nguyen}.
If $\tilde{B}_{2}$ is geometrically infinite, this contradicts  Remark~\ref{rem:flowlineintersectone}.
If $\tilde{B}_{2}$ is vertical, this contradicts the fact that $B_{2}$ is essential. 
If $\tilde{B}_{2}$ is geometrically finite, this contradicts Lemma~\ref{lem:geometrically finite parallel lines}.
\end{proof}

In the rest of this section, we equip $S$ with a hyperbolic metric $d_S$ such that the boundary (if nonempty) is totally geodesic. 

\begin{proof}[Proof of Theorem~\ref{thm:distortionalmostfiber}]
The non-geometric manifold $N$ has a finite cover such that each Seifert block in this cover is a trivial circle bundle over a hyperbolic surface (see Lemma~3.1 \cite{PW14}). We elevate $S \looparrowright N$ into this finite cover. By Proposition~\ref{prop:distortion2}, it suffices to prove the theorem in this cover. Thus, without loss of generality, we can assume that each Seifert block in $N$ is a trivial circle bundle over a hyperbolic surface.

Now we deal with the issue of metrics:
We always have a convenient metric (in the sense of Section~3.1)  on a mixed manifold . In the graph manifold case, we may pass to a further finite cover and hence assume that it is a simple graph manifold. Then we can choose a convenient metric on it as described in Section~3.1.



We first show that $f \preceq \Delta_{H}^{G}$. Every finitely generated subgroup of a surface group or free group is undistorted. It follows that for any component $S_i$ of $\Phi(S)$ then $\pi_1(S_{i})$ is undistorted in $\pi_1(S)$. 
It follows from Proposition~\ref{prop:distortion} that $\delta_{S_{i}}$ is dominated by the distortion of $\pi_1(S)$ in $\pi_1(N)$. Therefore, $f \preceq \Delta_{H}^{G}$.

We are now going to prove $\Delta_{H}^{G} \preceq \bar{f}$, which is less trivial.
Let $h \in H$ such that $d \bigl (\tilde{s}_{0},h(\tilde{s}_{0}) \bigr ) \le n$, we wish to show that $d_{\tilde{S}} \bigl (\tilde{s}_{0},h(\tilde{s}_{0}) \bigr )$ is bounded above by $\bar{f}(n)$. The theorem is proved by an application of Corollary~\ref{cor:GeometricDistortion}.
For each component $S_{i}$ of $\Phi(S)$, let $\tilde{\delta}_{S_i}$ be the distortion of $\tilde{S}_{i}$ in $\tilde{N}$. We note that $\tilde{\delta}_{S_i} \sim \delta_{S_i}$. Let \[
\tau(n) := \max \bigset{\tilde{\delta}_{S_{i}}(n)}{\text {$S_{i}$ is a component of $\Phi(S)$}}
\] and $\bar{\tau}$ is the superadditive closure of $\tau$. We note that $\bar{\tau} \sim \bar{f}$ by Remark~\ref{rem:superadditive}.

 We will assume that $\tilde{s}_0$ and $h(\tilde{s}_0)$ belong to distinct pieces of $\tilde{S}$, otherwise the fact $d_{\tilde{S}}(\tilde{s}_{0},h(\tilde{s}_{0}))$ is bounded above by $\bar{f}(n)$ is trivial. Without of generality, we assume that $s_{0}$ belongs to a curve in the collection $\mathcal{T}_{g}$. Let $\mathcal{Q}$ be the family of lines in $\tilde{S}$ that are lifts of curves of $\mathcal{T}_g$. We note that there are distinct lines $\ell$ and $\ell'$ in $\mathcal{Q}$ such that $\tilde{s}_0 \in \ell$ and $h(\tilde{s}_0) \in \ell'$. Let $e$ and $e'$ be the non-oriented edges in the tree $\mathbf{T}_{S}$ corresponding to the lines $\ell$ and $\ell'$ respectively. Choose the non backtracking path joining $e$ to $e'$ in the tree $\mathbf{T}_S$, with ordered vertices $v_0,v_1,\dots,v_{k-1}$ where $v_1$ is not a vertex on the edge $e$ and $v_{k-2}$ is not a vertex on the edge $e'$. We denote the pieces corresponding to the vertices $v_{i}$ by $\tilde{B}_{i}$ and the blocks corresponding to the vertices $\zeta(v_i)$ by $\tilde{M}_i$ with $i = 0,1, \cdots ,k-1$. We note that the blocks $\tilde{M}_i$ are distinct because $\zeta$ is injective by Proposition~\ref{prop:TreeBijective}.

For each piece $B$ of $S$, let $M$ be the block of $N$ such that $B$ is mapped into $M$. If $B \looparrowright M$ is vertical, let $R_{B}$ be the constant given by Lemma~\ref{lem:verticalgraphmanifold}. If  $B \looparrowright M$ is geometrically finite, we let $R_{B}$ be the constant given by Lemma~\ref{lem:geometricallyfinitehyper}. Since the number of vertical and geometrically finite pieces of $S$ is finite, we let $R$ be the maximum of the numbers $R_{B}$ chosen above. 

By a similar argument as in the proof of Theorem~6.1 in \cite{Hruska-Nguyen}, we can find a path $\gamma$ connecting $\tilde{s}_0$ to $h(\tilde{s}_0)$ that intersects each plane $\tilde{T}_{i} = \tilde{M}_{i-1} \cap \tilde{M}_{i}$ with $i=1,2,\dots ,k-1$ exactly at one point $y_{i}$ and satisfies $\abs{\gamma} \le Kd \bigl (\tilde{s}_{0},h(\tilde{s}_{0}) \bigr )$ where the constant $K$ depends only on the metric $d$. Here $|\cdot|$ denotes the length of a path with respect to the metric $d$.

If a piece $\tilde{B}_{i}$ is either vertical or geometrically finite in the corresponding block $\tilde{M}_{i}$, we let $\alpha_{i}$ be a path in $\tilde{M}_{i}$ connecting $y_i$ to $y_{i+1}$ and $\beta_{i}$ be a path in $\tilde{B}_{i}$ connecting a point in $\tilde{B}_{i} \cap \tilde{T}_{i}$ to a point in $\tilde{B}_{i} \cap \tilde{T}_{i+1}$ as given by Lemma~\ref{lem:verticalgraphmanifold} (when $\tilde{B}_{i}$ is vertical) and Lemma~\ref{lem:geometricallyfinitehyper} (when $\tilde{B}_{i}$ is geometrically finite). We replace $\beta_i$ by a geodesic in $\tilde{S}$ connecting $\beta_{i}(0)$ to $\beta_{i}(1)$. {\bf{By abuse of notation, we still denote this geodesic by $\beta_i$}}.

\begin{figure}[h]
\centering
\def\svgwidth{\columnwidth}
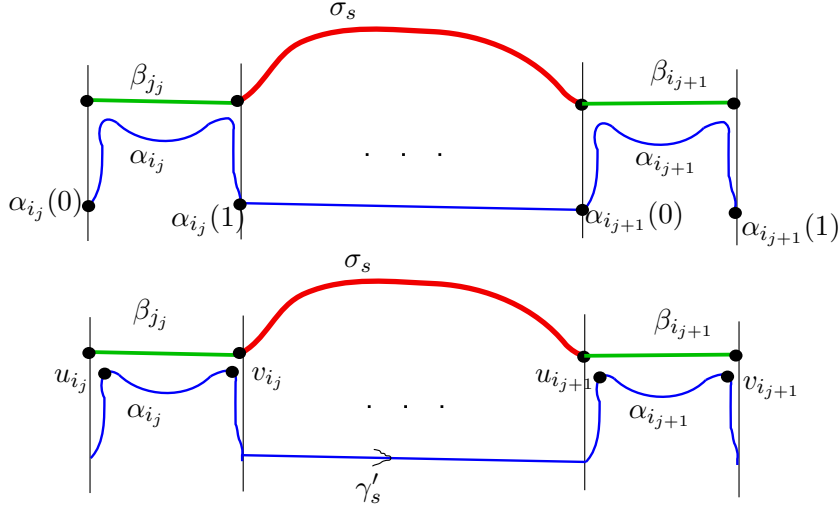
\caption{The upper picture illustrates the two end points of each path $\beta_{i_j}$ are within a $R$--neighborhood of $\alpha_{i_j} \subset \gamma'$. We add a geodesic in the almost fiber part of $\tilde{S}$ connecting $\beta_{i_j}(1)$ to $\beta_{i_{j+1}}(0)$.
The lower picture illustrates the point $v_{i_j}$ (resp.\,$u_{i_j}$) is within $R$--distance from $\beta_{i_j}(0)$ (resp.\,$\beta_{i_j}(1)$). The path $\gamma'_s$ is the subpath of $\gamma'$ that connects $v_{i_j}$ to $u_{i_{j+1}}$.}
\label{figure:1}
\end{figure}

On the path $\gamma$, every time the piece $\tilde{B}_{i}$ is either vertical or geometrically finite, we replace the subpath $\gamma_{|[y_{i},y_{i+1}]}$ of $\gamma$ by $\alpha_{i}$. We therefore obtain a new path denoted by $\gamma'$ such that \[
\abs{\gamma'} \le KRd \bigl (\tilde{s}_{0},h(\tilde{s}_{0}) \bigr ) \le KRn
\]

We now construct a path $\beta$ in $\tilde{S}$ connecting $\tilde{s}_{0}$ to $h(\tilde{s}_{0})$ which stays close to $\gamma'$ every time they both travel the same a block containing a piece which is either vertical or geometrically finite (see Figure~\ref{figure:1}). Let $\tilde{B}_{i_{0}},\dots,\tilde{B}_{i_{t}}$ be the collection of the vertical or geometrically finite pieces where $0 \le i_{0}\le \cdots \le i_{t} \le k-1$. From the given paths $\beta_{i_{0}},\dots,\beta_{i_{t}}$, we obtain a path $\beta$ in $\tilde{S}$ connecting $\tilde{s}_{0}$ to $h(\tilde{s}_{0})$ by adding a geodesic in $\tilde{S}$ connecting the endpoint of $\beta_{i_{j}}$ to the initial point of $\beta_{i_{j+1}}$ where $j$ varies from $0$ to $t-1$, adding a geodesic in $\tilde{S}$ connecting $\tilde{s}_{0}$ to the initial point of $\beta_{i_{0}}$, and a geodesic in $\tilde{S}$ connecting the endpoint of $\beta_{i_{t}}$ to $h(\tilde{s}_{0})$.

{\bf{Claim~1:}} There exists a linear function $J$ not depending on the choices of $n$, $h$, or $\beta$ such that $\sum_{j=0}^{t}\abs{\beta_{i_j}}_{\tilde{S}} \le J(n)$.

Since vertical pieces and geometrically finite pieces of $S$ are undistorted in the corresponding blocks of $N$, and there are finite many of vertical pieces and geometrically finite pieces, we have a constant $\epsilon >0$ such that the following holds. Suppose that B is either vertical piece or geometrically finite piece, then for any $x, y \in \tilde{B}$ we have $d_{\tilde{S}}(x,y) \le \epsilon d(x,y) + \epsilon$. 

We recall that $\alpha_{i_j}$ is a subpath of $\gamma'$ and $\beta_{i_j}$ is a geodesic in $\tilde{S}$. Since $\beta_{i_j}(0), \beta_{i_j}(1) \in \mathcal{N}_{R}(\alpha_{i_j})$, we have $d(\beta_{i_j}(0), \beta_{i_j}(1)) \le 2R + \abs{\alpha_{i_j}}$.
Let $\rho$ be the constant given by Remark~\ref{rem:constant rho}. We note that $k \le n/\rho$.

We have
\begin{align*}
   \sum_{j=0}^{t}\abs{\beta_{i_{j}}}_{\tilde{S}} &= \sum_{j=0}^{t} d_{\tilde{S}}(\beta_{i_j}(0),\beta_{i_j}(1)) \le \sum_{j=0}^{t} \Bigl ( \epsilon\, d(\beta_{i_j}(0),\beta_{i_j}(1)) + \epsilon \Bigr )\\
   &\le \sum_{j=0}^{t} \Bigl ( \epsilon (2R + \abs{\alpha_{i_j}}) +\epsilon \Bigr ) = \sum_{j=0}^{t}\epsilon \abs{\alpha_{i_j}} + \sum_{j=0}^{t}(2R\epsilon + \epsilon)\\
   &\le \epsilon \sum_{j=0}^{t} \abs{\alpha_{i_j}} +(t+1)(2R\epsilon + \epsilon) \le \epsilon \abs{\gamma'} + (t+1)(2R\epsilon + \epsilon)\\
   &\le \epsilon \abs{\gamma'} + (k+1)(2R\epsilon + \epsilon) \le \epsilon \abs{\gamma'} + (n/\rho+1)(2R\epsilon + \epsilon)\\
   & \le \epsilon K R n + (n/\rho+1)(2R\epsilon + \epsilon)
\end{align*}
Let $J(n) = \epsilon K R n + (n/\rho+1)(2R\epsilon + \epsilon)$, the claim is confirmed.

We consider the complement of $\beta - \cup_{j=0}^{t}\beta_{i_{j}}$, which can be written as a disjoint union of subpaths $\sigma_1,\dots,\sigma_{m}$ of $\beta$ with $m \le k$.

{\bf{Claim~2:}}
For each $i=1,\dots,m$, there exists a subpath $\gamma'_{i}$ of $\gamma'$ such that 
\[
d \bigl (\sigma_{i}(0),\sigma_{i}(1) \bigr ) \le 2R +\abs{\gamma'_{i}}
\]
and $\sum_{i=1}^{m}\abs{\gamma'_{i}} \le 3\abs{\gamma'}$.

Indeed, let $p(i) \neq q(i)$ be two numbers in the collection $\{i_{0},\dots, i_{t}\}$ such that $\sigma_i(0)$ is the endpoint of $\beta_{p(i)}$ and $\sigma_i(1)$ is the initial point of $\beta_{q(i)}$. For convenience, lets relabel $p =p(i)$, $q = q(i)$. Note that it is possible that the pieces $\tilde{B}_p$ and $\tilde{B}_q$ are adjacent.
Since $\beta_{p} \subset \mathcal{N}_{R}(\alpha_{p})$ and $\beta_{q} \subset \mathcal{N}_{R}(\alpha_{q})$, it follows that $d(\beta_{p}(1), v_{p}) \le R$ and $d(\beta_{q}(0), u_{q}) \le R$ for some $v_{p} \in \alpha_{p}$, $u_{q} \in \alpha_{q}$.

Let $\gamma'_{i}$ be the concatenation ${\alpha_{p}}_{|[v_{p},\alpha_{p}(1)]} \cdot 
\gamma'_{|[\alpha_{p}(1),\alpha_{q}(0)]} \cdot {\alpha_{q}}_{[\alpha_{q}(0),u_{q}]}$ (see Figure~\ref{figure:1}). It follows that $d(v_{p},u_{q}) \le \abs{\gamma'_i}$. Using $d(\beta_{p}(1),v_p) \le R$, $d(\beta_{q}(0), u_{q}) \le R$, $d(v_p, u_q) \le \abs{\gamma'_i}$, and the triangle inequality, we have
\begin{align*}
    d \bigl (\sigma_i(0),\sigma_i(1) \bigr ) &= d\bigl (\beta_{p}(1), \beta_{q}(0) \bigr )
    \le d\bigl (\beta_{p}(1), v_{p} \bigr ) + d\bigl (v_{p}, u_{q} \bigr ) + d\bigl (u_{q}, \beta_{q}(0)\bigr )\\
    &\le R + d(v_{p},u_{q}) +R
    =2R +d(v_{p},u_{q}) \le 2R + \abs{\gamma'_{i}}\\
\end{align*}
 Recall that $\alpha_{i_j}$ is a subpath of $\gamma'$ with $i_j \in \{i_0, \dots,i_t\}$ and $\sum_{j=0}^{t}\abs{\alpha_{i_j}} \le \abs{\gamma'}$. By the construction of $\gamma'_{i}$, we have \[
 \abs{\gamma'_{i}} \le \abs{\alpha_{p}} + \abs{\alpha_{q}} + \Bigl |\gamma'_{|[\alpha_{p}(1),\alpha_{q}(0)]} \Bigr |
 = \abs{\alpha_{p(i)}} + \abs{\alpha_{q(i)}} + \Bigl |\gamma'_{|[\alpha_{p(i)}(1),\alpha_{q(i)}(0)]} \Bigr|
\]
 Summing over $i$, we obtain
\begin{align*}
    \sum_{i=1}^{m}\abs{\gamma'_{i}} &\le \sum_{i=1}^{m} \Bigl (\abs{\alpha_{p(i)}} + \abs{\alpha_{q(i)}} \Bigr ) + \sum_{i=1}^{m} \Bigl | \gamma'_{|[\alpha_{p(i)}(1),\alpha_{q(i)}(0)]} \Bigr |\\
    &\le 2\sum_{j=0}^{t}\abs{\alpha_{i_j}} + \sum_{i=1}^{m} \Bigl |\gamma'_{|[\alpha_{p(i)}(1),\alpha_{q(i)}(0)]} \Bigr |
    \le 2\sum_{j=0}^{t}\abs{\alpha_{i_j}} + \abs{\gamma'} \le 2\abs{\gamma'} + \abs{\gamma'} =3\abs{\gamma'}.
\end{align*}
The claim is confirmed.

Let $S_{i}$ be the component of $\Phi(S)$ that the image of $\sigma_{i}$ under the covering map belongs to. We have the length of $\sigma_{i}$ in $\tilde{S}$ is no more than $\tilde{\delta}_{S_{i}} \bigl (2R+\abs{\gamma'_{i}} \bigr )$. Thus the sum of the lengths of $\sigma_{i}$ in $\tilde{S}$ is no more than \[
\tilde{\delta}_{S_1}\bigl (2R+\abs{\gamma'_{1}} \bigr ) + \cdots + \tilde{\delta}_{S_m}\bigl (2R + \abs{\gamma'_{m}} \bigr )
\]
which is less than or equal to $\tau \bigl (2Rm+\sum_{i=1}^{m}\abs{\gamma'_{i}} \bigr )$.
Since $\sum_{i=1}^{m}\abs{\gamma'_{i}} \le 3\abs{\gamma'} \le 3KRn$ and $m$ is bounded above by a linear function in term of n ($m \le k \le n/\rho$), it follows that $\abs{\beta}_{\tilde{S}} \preceq \bar{\tau}(n)$.
\end{proof}

\section{Distortion of clean almost fiber surfaces in mixed manifolds}
\label{sec:distortion mixed}
As we have shown in Section~\ref{sec:almostfiber}, distortion of a surface in a non-geometric $3$--manifold is determined by the distortion of components of the almost fiber part of the surface. We note that each component of the almost fiber part is a clean almost fiber surface. In this section, we compute the distortion of a clean almost fiber surface $S$ in $N$. The main theorem is the following.

\begin{thm}
\label{thm:UpperBound}
Let $g \colon (S,s_0)\looparrowright (N,x_0)$ be a clean almost fiber surface in a mixed manifold $N$. We assume that all Seifert fibered blocks of $N$ are non-elementary. Suppose that $S$ contains at least one geometrically infinite piece. Then the distortion of $\pi_1(S)$ in $\pi_1(N)$ is exponential if $\pi_1(S)$ is separable in $\pi_1(N)$, and double exponential if $\pi_1(S)$ is non-separable in $\pi_1(N)$.
\end{thm}

We recall that $\pi_1(S)$ is separable in $\pi_1(N)$ if and only if the spirality of $S$ is trivial (see Remark~\ref{rem:SUNLIU}). The proof of Theorem~\ref{thm:UpperBound} is divided into two parts. The proof of the lower bound of the distortion is given in Subsection~\ref{sub:1} and the proof of the upper bound of the distortion is given by Subsection~\ref{sub:2}.

\begin{setup}
\label{setup:1}
We equip $N$ with the metric $d$ as in Subsection~\ref{rem:metriconmixedmanifold}, and equip $S$ with a hyperbolic metric $d_S$ such that the boundary (if nonempty) is totally geodesic and the simple closed curves of $\mathcal{T}_{g}$ are geodesics. 

For each piece $B$ of $S$, let $M$ be the block of $N$ in which $B$ is mapped into $M$. By Remark~\ref{rem:virtual fiber}, there exists a finite cover $M_{B} \to M$ where $M_{B}$ is the mapping torus of a homeomorphism $\varphi$ of the surface $B$ such that $\varphi$ fixes periodic points on $\partial B$. 
Each boundary component $c$ of $B$ is mapped into a boundary torus of $M_{B}$, we fix a degeneracy slope on this torus, and denoted it by $\mathbf{s}_{cB}$. The pullback of the fibration $M_{B} \to S^1$ by the infinite cyclic covering map $\R \to S^1$ is $B \times \R$ (see the paragraph above Lemma~\ref{lem:height}), we identify the universal cover $\tilde{M}$ with $\tilde{B} \times \R$. We also assume that $\tilde{S} \cap \tilde{M} = \tilde{B} \times \{0\}$.
\end{setup}

\subsection{Upper bound of the distortion}
\label{sub:2}
In this subsection, we find the upper bound of the distortion of $\pi_1(S)$ in $\pi_1(N)$.
\begin{prop}
\label{prop:mixedupperbound}
The distortion of $\pi_1(S)$ in $\pi_1(N)$ is at most double exponential. Furthermore, if the spirality of $S$ is trivial then the distortion is at most exponential.
\end{prop}

We use the same strategy as in the upper bound section of \cite{Hruska-Nguyen} (see Section~6 of \cite{Hruska-Nguyen}) but techniques are different. We briefly discuss here the main difference between this current section and Section~6 in \cite{Hruska-Nguyen}. In the setting of graph manifold, a JSJ torus $T$ of $N$ receives two Seifert fibers from the blocks on both sides. In \cite{Hruska-Nguyen}, at any $y$ in $\tilde{T}$ (universal cover of $T$), we follow fibers (on both sides) until they meet $\tilde{S}$. Note that these fibers do not match up. In this current section, it is possible that one block containing $T$ is a Seifert fibered space and the other block containing $T$ is a hyperbolic block or both the blocks are hyperbolic, thus we will follow degeneracy slopes instead.
 Moreover, at $y \in \tilde{T}$, we need to be specific on which degeneracy slopes we should follow.

We describe here the outline of the proof of Proposition~\ref{prop:mixedupperbound}. For each $n \in \N$, let $h \in \pi_1(S,x_0)$ such that $d \bigl (\tilde{x}_{0},h(\tilde{x}_{0}) \bigr) \le n$. We would like to find an upper bound (either exponential or double exponential) of $d_{\tilde{S}} \bigl (\tilde{x}_{0},h(\tilde{x}_{0}) \bigr )$ in terms of $n$. Choose a path $\beta$ in $\tilde{N}$ connecting $\tilde{x}_{0}$ to $h(\tilde{x}_{0})$ with $\abs{\beta} \le n$ such that $\beta$ passes through a sequence of blocks $\tilde{M}_{0},\dots,\tilde{M}_{k}$, intersecting the plane $\tilde{T}_{j} = \tilde{M}_{j-1} \cap \tilde{M}_{j}$ exactly at one point that is denoted by $y_{j}$ with $j=1,\cdots,k$ if $k\ge 1$. There exists a piece $\tilde{B}_{j}$ of $\tilde{S}$ such that $\tilde{g}(\tilde{B}_{j}) \subset \tilde{M}_{j}$. Let $\rho$ be the constant given by Remark~\ref{rem:constant rho}. We note that $k \le n/\rho$.
Let $c_{j}$ be the circle in $\mathcal{T}_{g}$ that is universally covered by the line $\tilde{B}_{j-1} \cap \tilde{B}_{j}$ with $j=1,\cdots,k$. Let $\overleftarrow{\mathbf s_{j}} = \mathbf{s}_{c_{j}B_{j-1}}$ and $\overrightarrow{\mathbf s_{j}} = \mathbf{s}_{c_{j}B_{j}}$ be the degeneracy slopes in the corresponding tori $\overleftarrow{T'_{j}}$ and $\overrightarrow{T'_{j}}$ of the spaces $M_{B_{j-1}}$ and $M_{B_{j}}$ respectively (see Definition~\ref{defn:degeneracyslope}). The distortion function $\Delta$ of $\tilde{S}$ in $\tilde{N}$ does not change (up to equivalence in Definition~\ref{def:equivalentfunction}) when we add a linear function in term of $n$ to $\Delta$. Therefore, to make the argument simpler, using Corollary~\ref{cor:GeometricDistortion} and modifying $g$ by a homotopy, we may assume that the lifts of the degeneracy slopes and lines $\tilde{g}(\ell)$ (where $\ell$ is a line in $\mathcal{Q}$ the family of lines that are lifts of loops of $\mathcal{T}_{g}$) are straight lines in the corresponding planes of $\tilde{N}$. 
The line parallel to a lift of the degeneracy slope $\overleftarrow{\mathbf s_{j}}$ in $\tilde{T}_{j}$ passing through $y_{j}$ intersects $\tilde{g}(\tilde{S})$ in a unique point which is denoted by $x_{j}$. Similarly, the line parallel to a lift of the degeneracy slope $\overrightarrow{\mathbf s_{j}}$ in $\tilde{T}_{j}$ passing through $y_{j}$ intersect $\tilde{g}(\tilde{S})$ in one point which is denoted by $z_{j}$.

Similarly as in \cite{Hruska-Nguyen}, we show that $d_{\tilde{S}} \bigl (\tilde{x}_{0},h(\tilde{x}_{0}) \bigr )$ is dominated by the sum 
\[
e^{n}\sum_{j=1}^{k}e^{d(y_{j},x_{j}) + d(y_{j},z_{j})}
\] and we analyze the growth of the sequence 
\begin{multline*}
   d(y_1,x_1), d(y_1,z_1), d(y_2,x_2), \dots,\\
   d(y_{j-1},z_{j-1}),
   d(y_{j},x_{j}),
   d(y_j,z_j), \dots,\\
   d(y_{k},x_{k}),d(y_{k},z_{k})
\end{multline*}
An upper bound on $d(y_{j},x_{j})$ in terms of $d(y_{j-1},z_{j-1})$ will be described in Lemma~\ref{lem:Inequality1}. A relation between $d(y_{j},x_{j})$ and $d(y_{j},z_{j})$ will be Lemma~\ref{lem: application rule of sines}.

\begin{rem}
\label{rem:constants LandC}
\begin{enumerate}
\item It is possible from the construction above that $x_{j} = z_{j}$.
 \item Using Remark~\ref{rem:distortionwellknow}, the fact $S$ has only finite many pieces, and $N$ has finite many blocks, we obtain a constant $L \ge 1$ such that for each piece $\tilde{B}$ in $\tilde{S}$, we have $d_{\tilde{S}}(u,v) \le e^{Ld(u,v) + L}$ for any two points $u$ and $v$ in $\tilde{B}$.
\end{enumerate}

\end{rem}
Let $\overleftarrow{\lambda_{j}}$ and $\overrightarrow{\lambda_{j}}$  be the lengths of path lifts of the degeneracy slopes $\overleftarrow{\mathbf s_j}$ and $\overrightarrow{\mathbf s_j}$ in $\tilde{N}$ with respect to $d$--metric. Let $\overleftarrow{\ell_j}$ (resp.\, $\overrightarrow{\ell_j}$) be the line parallel to a lift of the degeneracy slope $\overleftarrow{\mathbf s_j}$ (resp.\, $\overrightarrow{\mathbf s_j}$) in $\tilde{T}_j$ such that $y_j \in \overleftarrow{\ell_j}$ (resp.\, $y_j \in \overrightarrow{\ell_j}$). We note that $\overleftarrow{\ell_j}$ intersects $\tilde{S}$ at $x_j$, and $\overrightarrow{\ell_j}$ intersects $\tilde{S}$ at $z_j$.

\begin{lem}[Crossing a block]
\label{lem:Inequality1}
There exists a positive constant $L'$ such that the following holds: For any $j =1,\cdots,k$
\[d(y_{j},x_{j}) \le \frac{\overleftarrow{\lambda_{j}}}{\overrightarrow{\lambda_{j-1}}} d(y_{j-1},z_{j-1}) + L'\,d(y_{j},y_{j-1})
\]
\end{lem}

\begin{proof}
We recall that the finite covering space $M_{B_{j-1}}$ of $M_{j-1}$ is fibered over circle with the fiber $B_{j-1}$, and the block $\tilde{M}_{j-1}$ is identified with $ \tilde{B}_{j-1} \times \R$.

We recall that the line $\overrightarrow{\ell_{j-1}}$ parallel to a lift of the degeneracy slope $\overrightarrow{\mathbf s_{j-1}}$ and $\overrightarrow{\ell_{j-1}}$ passes through $y_{j-1}$ and $z_{j-1}$.
On the line $\overrightarrow{\ell_{j-1}}$, choose a point $u_{j-1}$ such that $u_{j-1} \in \tilde{B}_{j-1} \times \{n\}$ for some integer $n$, $d(u_{j-1},y_{j-1}) \le \overrightarrow{\lambda_{j-1}}$ and
 \begin{equation}
 \label{eq:1}
d(u_{j-1},z_{j-1}) \le  d(y_{j-1},z_{j-1})
\end{equation}

Similarly, on the line $\overleftarrow{\ell_{j}}$, choose a point $v_{j} \in \tilde{B}_{j-1} \times \{m\}$ for some integer $m$ and $v_{j}$ such that $d(v_{j},y_{j}) \le \overleftarrow{\lambda_{j}}$. It follows that
 \begin{equation}
 \label{eq:2}
 \begin{split}
     d(u_{j-1},v_{j}) &\le d(u_{j-1},y_{j-1}) +d(y_{j-1},y_{j}) + d(y_{j},v_{j})\\
     &\le \overrightarrow{\lambda_{j-1}} + d(y_{j-1},y_{j}) + \overleftarrow{\lambda_{j}}
 \end{split}
 \end{equation}
 
Let $\rho >0$ be the constant given by Remark~\ref{rem:constant rho}. Let $L$ and $C$ be contants given by Lemma~\ref{lem:height}. We use Lemma~\ref{lem:height} and the fact $\rho \le d(u_{j-1},v_{j})$ to see that
\begin{equation}
\label{eq:oplus}
    \begin{split}
    \abs{m-n} &\le L\,d(u_{j-1},v_{j}) + C \\
    &\le Ld(u_{j-1},v_{j}) + \frac{C}{\rho}d(u_{j-1},v_{j}) \\
    &= (L+ \frac{C}{\rho})d(u_{j-1},v_{j})
\end{split}
\end{equation}
Let $L_{j} = \overleftarrow{\lambda_{j}}\, (L + C/\rho) + \overleftarrow{\lambda_{j}}/\rho + \bigl ((\overleftarrow{\lambda_{j}})^{2}
+ \overrightarrow{\lambda_{j-1}}\,\overleftarrow{\lambda_{j}}\bigr )(L + C/\rho)(1/\rho)$.

We use \eqref{eq:1}, \eqref{eq:2}, \eqref{eq:oplus} and the facts $d(y_j,v_j) \le \overleftarrow{\lambda_{j}}$, $d(v_j,x_j) = \abs{m}\overleftarrow{\lambda_{j}}$ and $1 \le d(y_{j-1},y_j) /\rho$  to see that
\begin{align*}
    d(y_j,x_j) &\le d(y_j,v_j) + d(v_j,x_j)
    \le \overleftarrow{\lambda_{j}} + d(v_j,x_j)
    \le \overleftarrow{\lambda_{j}} + \abs{m}\overleftarrow{\lambda_{j}} \\
    &\le \overleftarrow{\lambda_{j}} + \abs{n}\overleftarrow{\lambda_{j}} + \abs{m-n}\overleftarrow{\lambda_{j}}\\
    &= \overleftarrow{\lambda_{j}} + \frac{\overleftarrow{\lambda_{j}}}{\overrightarrow{\lambda_{j-1}}}d(u_{j-1},z_{j-1}) + \abs{m-n}\overleftarrow{\lambda_{j}} \\
    &\le \overleftarrow{\lambda_{j}} + \frac{\overleftarrow{\lambda_{j}}}{\overrightarrow{\lambda_{j-1}}}d(u_{j-1},z_{j-1}) + (L+\frac{C}{\rho})\overleftarrow{\lambda_j}d(u_{j-1},v_{j}) \quad \text{by \eqref{eq:oplus}}\\
    &\le \overleftarrow{\lambda_{j}} + \frac{\overleftarrow{\lambda_{j}}}{\overrightarrow{\lambda_{j-1}}}d(y_{j-1},z_{j-1}) + (L+\frac{C}{\rho})\overleftarrow{\lambda_j}d(u_{j-1},v_{j}) \\
    &\le \overleftarrow{\lambda_{j}} + \frac{\overleftarrow{\lambda_{j}}}{\overrightarrow{\lambda_{j-1}}}d(y_{j-1},z_{j-1}) + (L+\frac{C}{\rho})\overleftarrow{\lambda_j}\Bigl (\overrightarrow{\lambda_{j-1}} + d(y_{j-1},y_{j}) + \overleftarrow{\lambda_{j}} \Bigr )\\
    &=\frac{\overleftarrow{\lambda_{j}}}{\overrightarrow{\lambda_{j-1}}}d(y_{j-1},z_{j-1}) + \overleftarrow{\lambda_{j}} + (L+\frac{C}{\rho})\overleftarrow{\lambda_j}\Bigl (\overrightarrow{\lambda_{j-1}} + \overleftarrow{\lambda_{j}} \Bigr ) + (L+ \frac{C}{\rho})\overleftarrow{\lambda_j}d(y_{j-1},y_j)\\
    &\le \frac{\overleftarrow{\lambda_{j}}}{\overrightarrow{\lambda_{j-1}}}d(y_{j-1},z_{j-1}) + \Bigl ( \overleftarrow{\lambda_{j}} + (L+\frac{C}{\rho})\overleftarrow{\lambda_j}\bigl (\overrightarrow{\lambda_{j-1}} + \overleftarrow{\lambda_{j}} \bigr ) \Bigr )\frac{d(y_{j-1},y_j)}{\rho} \\
    &+ (L+ \frac{C}{\rho})\overleftarrow{\lambda_j}d(y_{j-1},y_j)\\
    &= \frac{\overleftarrow{\lambda_{j}}}{\overrightarrow{\lambda_{j-1}}} d(y_{j-1},z_{j-1}) + L_{j}\,d(y_{j},y_{j-1})
\end{align*}
Because there are only finitely many pieces of $S$ and blocks of $N$, we can choose a constant $L'$ (may be maximum of all possible constants $L_j$) that is large enough to satisfy the conclusion of the lemma.
\end{proof}

\begin{lem}[Crossing a JSJ plane]
\label{lem: application rule of sines}
\[
\frac{d(y_{j},z_{j})}{d(y_{j},x_{j})} = \xi_{j} \cdot \frac{\overrightarrow{\lambda_{j}}}{\overleftarrow{\lambda_{j}}}
\]
\end{lem}

\begin{proof}
Choose non-zero integers $n$ and $m$ such that the slice $\tilde{B}_{j-1} \times \{n\}$ of $\tilde{M}_{j-1} = \tilde{M}_{B_{j-1}}$ is glued into the slice $\tilde{B_{j}} \times \{m\}$ of $\tilde{M}_{j} = \tilde{M}_{B_{j}}$. Choose a point $y'$ in $\tilde{B}_{j-1} \times \{n\}$, and two points $x'$ and $z'$ in $\tilde{S} \cap \tilde{T}_{j}$ such that $[y',z']$ and $[y_{j},z_{j}]$ are parallel segments as well as $[y',x']$ and $[y_{j},x_{j}]$ are parallel segments. Since $\Delta(x',y',z')$ and $\Delta(x_{j},y_{j},z_{j})$ are similar triangles, it follows that 
\[
d(y',z') \bigl / d(y',x') = d(y_{j},z_{j}) \bigl / d(y_{j},x_{j})
\]
Thus, without loss of generality, we may assume that $y_{j}$ belongs to the slice $\tilde{B_{j}} \times \{m\}$, and $y_{j}$ belongs to the slice $\tilde{B}_{j-1} \times \{n\}$. We note that
\[
 \abs{n} \bigl [\overleftarrow{T'_{j}}:\overleftarrow{T_{j}} \bigr ] = \abs{m} \bigl [\overrightarrow{T'_{j}}:\overrightarrow{T_{j}} \bigr ] 
\]
Thus
\[
\abs{m} \bigl /\abs{n}  =  \bigl [\overleftarrow{T'_{j}}:\overleftarrow{T_{j}} \bigr ] \bigl / \bigl [\overrightarrow{T'_{j}}:\overrightarrow{T_{j}} \bigr ] = \xi_{j}
\]
Since $d(y_{j},z_{j}) = \abs{m}\,\overrightarrow{\lambda_{j}}$ and $d(y_{j},x_{j}) = \abs{n}\,\overleftarrow{\lambda_{j}}$, we have
\[
\frac{d(y_{j},z_{j})}{d(y_{j},x_{j})} = \xi_{j} \cdot \frac{\overrightarrow{\lambda_{j}}}{\overleftarrow{\lambda_{j}}}
\]
\end{proof}

\begin{proof}[Proof of Proposition~\ref{prop:mixedupperbound}]
We assume that the base point $s_0$ belongs to a curve in the collection $\mathcal{T}_g$.
For any $h \in \pi_1(S,s_0)$ such that $d \bigl (\tilde{s}_0,h(\tilde{s}_0) \bigr ) \le n$, we will show that  $d_{\tilde{S}}\bigl(\tilde{s}_0,h(\tilde{s}_0) \bigr)$ is bounded above by a double exponential function in terms of $n$. Let $L$ be the constant given by Remark~\ref{rem:constants LandC}, and let $L'$ be the constant given by Lemma~\ref{lem:Inequality1}. We consider the following cases:

{\bf Case 1:} 
$\tilde{s}_0$ and $h(\tilde{s}_0)$ belong to the same a piece $\tilde{B}$.
By Remark~\ref{rem:distortionwellknow} then  $d_{\tilde{S}}\bigl (\tilde{s}_0,h(\tilde{s}_0) \bigr) \le e^{Ln + L}$ which is dominated by an exponential function.

{\bf Case 2:} 
$\tilde{s}_0$ and $h(\tilde{s}_0)$ belong to distinct pieces of $\tilde{S}$.
Let $y_{j}$, $x_{j}$, and $z_{j}$ be points described as in the previous paragraphs. For convenience, relabel $\tilde{s}_{0}$ by $y_{0}$, and $h(\tilde{s}_{0})$ by $y_{k+1}$.

{\bf Claim 1:} Let $z_{0} = y_{0}$ and $z_{k+1} =y_{k+1}$. We have the following inequality.
\[d_{\tilde{S}}\bigl (\tilde{s}_0,h(\tilde{s}_0)\bigr ) \le e^{Ln}\sum_{j=0}^{k}e^{Ld(z_j,y_j) +Ld(z_{j+1},y_{j+1}) + L}
\]
We write $\beta = \beta_{0}\cdot \beta_{1} \cdots \beta_{k}$ where $\beta_j$ is the subpath of $\beta$ in $\tilde{M}_j$ connecting $y_j$ to $y_{j+1}$. For each $j = 0,\dots,k$ we have
\begin{align*}
   d(z_{j},z_{j+1}) &\le d(z_{j},y_{j}) + d(y_{j},y_{j+1}) + d(y_{j+1}, z_{j+1}) \\
   &\le d(z_{j},y_{j}) + \abs{\beta_{j}} + d(y_{j+1}, z_{j+1})
   \end{align*}
Using Remark~\ref{rem:constants LandC} we obtain
\begin{align*}
  d_{\tilde{S}}(\tilde{s}_{0},h(\tilde{s}_0)) =   d_{\tilde{S}}(y_{0},y_{k+1}) &\le \sum_{j=0}^{k} d_{\tilde{S}}(z_{j},z_{j+1}) \le \sum_{j=0}^{k}e^{Ld(z_{j},z_{j+1}) + L}\\
    &\le \sum_{j=0}^{k}e^{L\bigl (d(z_j,y_j) + \abs{\beta_j} + d(z_{j+1},y_{j+1}) \bigr ) + L}\\
    &\le \sum_{j=0}^{k}e^{L\bigl (d(z_j,y_j) + n + d(z_{j+1},y_{j+1}) \bigr ) + L}\\
    &\le e^{Ln}\sum_{j=0}^{k}e^{L\bigl (d(z_j,y_j) +d(z_{j+1},y_{j+1}) \bigr ) + L}
\end{align*}
Claim~1 is confirmed.


 We note that if $F(n) \sim e^{e^{n}}$, and $E(n) \sim e^{n}$ then $e^{n} F(n) \sim e^{e^{n}}$ and $e^{n} E(n) \sim e^{n}$. To complete the proof of the proposition, it suffices to find an appropriate upper bound of the sum appearing in Claim~1 which is  a double exponential function in general, and exponential function when the spirality of $S$ is trivial.  
 
 By  Lemma~\ref{lem:Inequality1}, we have 
\begin{equation}
\tag{$*$}
\label{eq:blacklozenge}
d(y_{j},x_{j}) \le \frac{\overleftarrow{\lambda_{j}}}{\overrightarrow{\lambda_{j-1}}} d(y_{j-1},z_{j-1}) + L'\,\abs{\beta_{j-1}}
\end{equation}
By Lemma \ref{lem: application rule of sines}, we have
\begin{equation}
\tag{$\dagger$}
\label{eq:ddagger}
   d(y_j,z_j) = \xi_{j} \cdot \frac{\overrightarrow{\lambda_{j}}}{\overleftarrow{\lambda_{j}}}\,d(y_j,x_j)
\end{equation}

{\bf Claim 2:} Suppose that the spirality of $S$ is non-trivial. There exists a function $F$ not depending on $\beta$, $n$, and $h$ such that
\[
\sum_{j=0}^{k}e^{Ld(z_j,y_j) +Ld(z_{j+1},y_{j+1})  + L} \le F(n)
\] and $F(n) \sim e^{e^{n}}$

Let $\epsilon$ be the governor of $g$ with respect to the chosen mapping tori (see Definition~\ref{defn:DilationSlopes}). Since the spirality of $S$ is non-trivial, it follows that $\epsilon $ is strictly greater than $1$.
Let $\mathcal{D}$ be the collection of the degeneracy slopes given by Set up~\ref{setup:1}. We note that $\mathcal{D}$ is a finite collection. For each degeneracy slope $\mathbf{s}_{cB}$, let $\abs{\mathbf{s}_{cB}}$ be the length of path lift of $\mathbf{s}_{cB}$ in $\tilde{N}$ with respect to $d$--metric. Let $\delta$ be the maximum of all possible ratios $\abs{\mathbf{s}_{cB}} / \abs{\mathbf{s}_{c'B'}}$.

We will show that for each $j = 0,\dots, k$ then
\begin{equation}
\label{eq:claim2}
    d(y_j, z_j) \le \frac{L'\delta n}{\epsilon-1}\epsilon^{j+1}
\end{equation}
To see (\ref{eq:claim2}), we first show by induction on $j =0,\dots, k$ that 
\[
d(y_j,z_j) \le L'n\sum_{i=1}^{j}\frac{\overrightarrow{\lambda_{j}}}{\overleftarrow{\lambda_{j+1-i}}}\epsilon^{i}
\]
The base case is trivial since $y_0 =z_0$, so both sides of the inequality equal zero. For inductive step, we use \eqref{eq:blacklozenge}, \eqref{eq:ddagger}, the inequality $\xi_j \le \epsilon$, and the fact $\abs{\beta_{j-1}} \le n$  to see that
\begin{align*}
    d(y_j,z_j) &= \xi_{j}\,\frac{\overrightarrow{\lambda_{j}}}{\overleftarrow{\lambda_{j}}}d(y_j,x_j) \le \epsilon\, \frac{\overrightarrow{\lambda_{j}}}{\overleftarrow{\lambda_{j}}}d(y_j,x_j) \le \epsilon\, \frac{\overrightarrow{\lambda_{j}}}{\overleftarrow{\lambda_{j}}} \Bigl ( \frac{\overleftarrow{\lambda_{j}}}{\overrightarrow{\lambda_{j-1}}}d(y_{j-1},z_{j-1}) + L'\abs{\beta_{j-1}} \Bigr )\\
    &\le \epsilon\, \frac{\overrightarrow{\lambda_j}}{\overrightarrow{\lambda_{j-1}}} d(y_{j-1},z_{j-1}) + \epsilon\, \frac{\overrightarrow{\lambda_{j}}}{\overleftarrow{\lambda_{j}}} L' \abs{\beta_{j-1}}\\
    &\le \epsilon\, \frac{\overrightarrow{\lambda_j}}{\overrightarrow{\lambda_{j-1}}} d(y_{j-1},z_{j-1}) + \epsilon\, \frac{\overrightarrow{\lambda_{j}}}{\overleftarrow{\lambda_{j}}} L'n\\
    &\le  \epsilon\, \frac{\overrightarrow{\lambda_j}}{\overrightarrow{\lambda_{j-1}}} L'n \sum_{i=1}^{j-1}\frac{\overrightarrow{\lambda_{j-1}}}{\overleftarrow{\lambda_{j-i}}}\epsilon^{i} + \epsilon\, \frac{\overrightarrow{\lambda_{j}}}{\overleftarrow{\lambda_{j}}} L'n\\
    &= L'n\sum_{i=1}^{j-1}\frac{\overrightarrow{\lambda_{j}}}{\overleftarrow{\lambda_{j-i}}}\epsilon^{i+1} + \epsilon\, \frac{\overrightarrow{\lambda_{j}}}{\overleftarrow{\lambda_{j}}} L'n = L'n \Bigl (\sum_{i=1}^{j-1}\frac{\overrightarrow{\lambda_{j}}}{\overleftarrow{\lambda_{j-i}}}\epsilon^{i+1} + \epsilon\, \frac{\overrightarrow{\lambda_{j}}}{\overleftarrow{\lambda_{j}}} \Bigr )\\
    &= L'n \Bigl (\sum_{i=2}^{j}\frac{\overrightarrow{\lambda_{j}}}{\overleftarrow{\lambda_{j+1-i}}}\epsilon^{i} + \epsilon\, \frac{\overrightarrow{\lambda_{j}}}{\overleftarrow{\lambda_{j}}} \Bigr ) = 
     L'n\sum_{i=1}^{j}\frac{\overrightarrow{\lambda_{j}}}{\overleftarrow{\lambda_{j+1-i}}}\epsilon^{i}
\end{align*}
Since $\overleftarrow{\lambda_j} \bigl / \overleftarrow{\lambda_{j+1-i}}$ is bounded above by $\delta$, it follows that
\[
d(y_j,z_j) \le L'n\delta \sum_{i=1}^{j}\epsilon^{i} \le L'n\delta \epsilon^{j+1} \bigl/ (\epsilon-1),
\]
establishing (\ref{eq:claim2}).
Summing over $j$, we obtain
\[
\sum_{j=1}^{k}d(y_j,z_j) \le \frac{L'\delta n}{\epsilon-1} \bigl (\epsilon^{2} + \cdots + \epsilon^{k+1}) \le  \frac{L'\delta n \epsilon^{2}}{(\epsilon-1)^{2}} \epsilon^{k+2} \le \frac{L'\delta n \epsilon^{2}}{(\epsilon-1)^{2}} \epsilon^{n/\rho+2}
\]
which is equivalent to an exponential function of $n$.
We use the facts $y_0 =z_0$, $y_{k+1} =z_{k+1}$, $L \ge 1$, and the fact that $e^{x}$ is superadditive on $[1,\infty)$ to see that
\begin{align*}
    \sum_{j=0}^{k}e^{Ld(z_j,y_j) +Ld(z_{j+1},y_{j+1}) + L} &\le e^{\sum_{j=0}^{k}Ld(z_j,y_j) +Ld(z_{j+1},y_{j+1}) + L}\\
    &= e^{(k+1)L + \sum_{j=0}^{k}Ld(z_j,y_j) +Ld(z_{j+1},y_{j+1})}\\
    &= e^{(k+1)C + 2L\sum_{j=1}^{k}d(y_j,z_j)}\\
    & \le e^{(n/\rho +1) + 2L\sum_{j=1}^{k}d(y_j,z_j)}
\end{align*}
which is equivalent to a double exponential function (because $\sum_{j=1}^{k}d(y_j,z_j)$ is equivalent to an exponential function of $n$, so $(n/\rho +1) + 2L\sum_{j=1}^{k}d(y_j,z_j)$ is also equivalent to an exponential function of $n$). Claim~2 is confirmed.

{\bf Claim 3:} Suppose the spirality of $S$ is trivial. There exists a function $E$ not depending on $\beta$, $n$, and $h$ such that
\[
\sum_{j=1}^{k}e^{Ld(y_{j},z_{j}) + Ld(y_{j+1},z_{j+1}) + L} \le E(n)
\] and $E(n) \sim e^{n}$.

For any $1 \le i \le j$, let $\Theta_{i,j}= \xi_{i} \,\xi_{i+1} \cdots \xi_{j}$.
Let $\Lambda$ be the constant given by Proposition~\ref{prop:virtualupper}. In order to prove Claim~3, we follow the argument in the proof of Claim~3 of Theorem~6.1 in \cite{Hruska-Nguyen}. We note that 
\[
\Theta_{i,j-1}\,\xi_{j} = \Theta_{i,j}
\]
We will show by induction on $j =0,1,\dots, k$ that
\[
d(y_j,z_j) \le L'\sum_{i=1}^{j}\abs{\beta_{i-1}}\frac{\overrightarrow{\lambda_{j}}}{\overleftarrow{\lambda_i}}\,\Theta_{i,j}
\]
The base case $j=0$ is trival since both sides of the inequality equal zero. For the inductive step, we use \eqref{eq:blacklozenge}, \eqref{eq:ddagger} to see that
\begin{align*}
    d(y_j,z_j) &= \xi_{j}\,\frac{\overrightarrow{\lambda_j}}{\overleftarrow{\lambda_j}}d(y_j,x_j) = d(y_j,x_j)\frac{\overrightarrow{\lambda_j}}{\overleftarrow{\lambda_j}}\,\xi_{j}\\
    &\le \Bigl (\frac{\overleftarrow{\lambda_j}}{\overrightarrow{\lambda_{j-1}}}\,d(y_{j-1},z_{j-1}) +L'\abs{\beta_{j-1}}\Bigr)\,\frac{\overrightarrow{\lambda_j}}{\overleftarrow{\lambda_j}}\,\xi_{j}\\
    &=d(y_{j-1},z_{j-1})\frac{\overrightarrow{\lambda_j}}{\overrightarrow{\lambda_{j-1}}}\,\xi_{j} + L'\abs{\beta_{j-1}}\frac{\overrightarrow{\lambda_j}}{\overleftarrow{\lambda_j}}\,\xi_j\\
    &\le \Bigl (L'\sum_{i=1}^{j-1}\abs{\beta_{i-1}}\,\frac{\overrightarrow{\lambda_{j-1}}}{\overleftarrow{\lambda_i}}\,\Theta_{i,j-1} \Bigr)\,\frac{\overrightarrow{\lambda_j}}{\overrightarrow{\lambda_{j-1}}}\,\xi_{j} + L'\abs{\beta_{j-1}}\frac{\overrightarrow{\lambda_j}}{\overleftarrow{\lambda_j}}\,\xi_j\\
    &= L'\sum_{i=1}^{j-1}\abs{\beta_{i-1}}\frac{\overrightarrow{\lambda_j}}{\overleftarrow{\lambda_i}}\,\Theta_{i,j-1}\,\xi_j + L'\abs{\beta_{j-1}}\frac{\overrightarrow{\lambda_j}}{\overleftarrow{\lambda_j}}\,\xi_j\\
    &= L'\sum_{i=1}^{j-1}\abs{\beta_{i-1}}\frac{\overrightarrow{\lambda_j}}{\overleftarrow{\lambda_i}}\,\Theta_{i,j} + L'\abs{\beta_{j-1}}\frac{\overrightarrow{\lambda_j}}{\overleftarrow{\lambda_j}}\,\Theta_{j,j}\\
    &= L'\sum_{i=1}^{j}\abs{\beta_{i-1}}\frac{\overrightarrow{\lambda_j}}{\overleftarrow{\lambda_i}}\,\Theta_{i,j}
\end{align*}
Since $\Theta_{i,j}$ is bounded above by $\Lambda$, $\overrightarrow{\lambda_j} / {\overleftarrow{\lambda_i}}$ is bounded above by $\delta$, and $\sum_{i=1}^{j}\abs{\beta_{i-1}} \le \abs{\beta} \le n$, we have
\[d(y_j, z_j) \le L'\sum_{i=1}^{j}\abs{\beta_{i-1}}\delta \Lambda = L'\delta \Lambda \sum_{i=1}^{j}\abs{\beta_{i-1}} \le L' \delta \Lambda n\]
It follows that \[e^{Ld(y_j,z_j) + Ld(y_{j+1},z_{j+1}) + L} \le e^{ 2LL'\delta \Lambda n + L}\]
Summing over $j$, we obtain
\begin{align*}
    \sum_{j=0}^{k}e^{Ld(y_j,z_j) + Ld(y_{j+1},z_{j+1}) + L} &\le \sum_{j=0}^{k}e^{ 2LL'\delta \Lambda n + L} = (k+1)e^{ 2LL'\delta \Lambda n + L} \\
    &\le (n/\rho +1)e^{ 2LL'\delta \Lambda n + L}
\end{align*}
which is equivalent to an exponential function of $n$, establishing Claim~3.

If the spirality of $S$ is non-trivial, Claim~1 combines with Claim~2 gives a double exponential upper bound for $d_{\tilde{S}} \bigl (\tilde{s}_{0},h(\tilde{s}_{0}) \bigr )$. If the spirality of $S$ is trivial, we combine Claim~1 and Claim~3 to get an exponential upper bound. The proposition follows from Corollary~\ref{cor:GeometricDistortion}.
\end{proof}
\subsection{Lower bound of the distortion}
\label{sub:1}
In this subsection, we compute the lower bound of the distortion of $\pi_1(S)$ in $\pi_1(N)$.

\begin{prop}
The distortion of $\pi_1(S)$ in $\pi_1(N)$ is at least exponential.
\end{prop}
\begin{proof}
We recall that $S$ contains a geometrically infinite piece. The fundamental group of the geometrically infinite piece is exponentially distorted in the fundamental group of the corresponding hyperbolic block of $N$ (see Remark~\ref{rem:distortionwellknow}). The fundamental group of the geometrically infinite piece is undistorted in $\pi_1(S)$ (in fact, every finitely generated subgroup of $\pi_1(S)$ is undistorted), and the fundamental group of the hyperbolic block is undistorted in $\pi_1(N)$. We combine these facts and Proposition~\ref{prop:distortion} to get the proof of this proposition.
\end{proof}

For the rest of this subsection, we will compute the lower bound (double exponential) of the distortion of $\pi_1(S)$ in $\pi_1(N)$ when the spirality of $S$ is non-trivial. 

\begin{prop}
\label{prop:lower}
The distortion $\pi_1(S)$ in $\pi_1(N)$ is at least double exponential if the spirality of $S$ is non-trivial.
\end{prop}
{\bf{The Goal}}: Let $\tilde{s}_0$ be a lift of $s_0$ in $\tilde{S}$. For convenience, we label $\tilde{s}_{0}$ by $z_1$. Our goal in this section is to construct a sequence of elements $\{z_{n}\}$ in $\tilde{S}$ such that $d(z_{1},z_{n}) \le n$ and $d_{\tilde{S}}(z_{1},z_{n})$ is bounded from below by a double exponential function in terms of $n$.

\begin{lem}
\label{lem:construct sequence}
Let $\gamma$ be a geodesic loop in $S$ such that $\gamma$ and $\mathcal{T}_{g}$ have nonempty intersection and such that $w(\gamma) >1$. There exists a positive number $A = A(\gamma)$ such that the following holds: Let $\{c_1,\dots,c_m\}$ be the sequence of curves of $\mathcal{T}_{g}$ (see Definition~\ref{defn:clean surface}) crossed by $\gamma$. The image of the circle $g(c_i)$ in $M$ lies in a JSJ torus $T_i$ obtained by gluing to a boundary torus $\overleftarrow{T_i}$ of $M_{i-1}$ to a boundary torus $\overrightarrow{T_i}$ of $M_{i}$. Let $\overleftarrow{T'_{i}}$ and $\overrightarrow{T'_{i}}$ be the boundary tori of $M_{B_{i-1}}$ and $M_{B_{i}}$ where the circle $c_{i}$ is embedded into $\overleftarrow{T'_{i}}$ and $\overrightarrow{T'_{i}}$ respectively. For each $i= 1,2,\dots,m$, let $\xi_{i} = \bigl [\overleftarrow{T_{i}'}:\overleftarrow{T_{i}} \bigr ] \bigl / \bigl [\overrightarrow{T_{i}'}:\overrightarrow{T_{i}}]$. 
Extend the sequence $\xi_1,\dots,\xi_m$ to a periodic sequence $\{\xi_j\}_{j=1}^\infty$ with $\xi_{j+m} = \xi_j$ for all $j>0$. Then there exists a \textup{(}nonperiodic\textup{)} sequence of integers $\bigl\{ t_{j} \bigr\}_{j=1}^\infty$, depending on our choice of the loop $\gamma$ such that
\begin{enumerate}
    \item
    \label{item:item:sequence:ApproxUpper}
    $0 \le t_{j} / \xi_{j} - t_{j-1} \le A $
 and $ t_{j} / \xi_{j} \in \N$ for all $j \ge 2$.
 \item 
 \label{item:item:sequence exponentialgrowth}
 $t_{nm+1} \ge t_{1}\,w(\gamma)^{n}$ for all $n \ge 1$.
 \item 
 \label{item:item:sequenceexponentialbelow}Let $\epsilon$ be the governor of $g$ with respect to the chosen mapping tori (see Definition~\ref{defn:DilationSlopes}). There exists a positive constant $D$ depending only on $A$, $\epsilon$, and $t_1$ such that $t_{n} \le e^{Dn +D}$ for all $n \ge 1$.
\end{enumerate}
\end{lem}

\begin{proof}
We recall that the spirality of $\gamma$ is the number $w(\gamma) = \xi_{1}\xi_{2}\cdots \xi_{m}$ (see Definition~\ref{defn:DilationSlopes}). For each $j \in \{1,\dots, m\}$, let $p_j = [\overleftarrow{T'_{j}}: \overleftarrow{T_j}]$ and $q_{j} = [\overrightarrow{T'_{j}}: \overrightarrow{T_j}]$. Extend the sequence $p_1,p_2,\dots, p_m$ to a $m$--periodic sequence $\{p_j\}_{j=1}^{\infty}$ with $p_{j+m} = p_j$ for all $j \ge 1$, and similarly extend $q_1,\dots, q_m$ to an $m$--periodic sequence $\{q_j\}_{j=1}^{\infty}$. Let $A =  \max \bigset{1+ q_j}{j=1,2,\dots ,m}$.
Let $t(1) = q_{1}q_{2}\cdots q_{m}$, we construct an infinite sequence $\{t_j\}$ satisfying (\ref{item:item:sequence:ApproxUpper}). Suppose that $t_{j-1}$ has been defined for some $j \ge 2$, and we would like to define $t_j$. Since $1 +q_j \le A$, we have
\[
1 \le \frac{A-1}{q_j} = \frac{A +t_{j-1} -1 -t_{j-1}}{q_j} = \frac{A+t_{j-1}}{q_j} - \frac{1+t_{j-1}}{q_j}
\]
It follows that there exists $k_j \in \N$ such that
\begin{equation}
    \tag{$\clubsuit$}
    \label{eqn:double:inequality}
    \frac{1+t_{j-1}}{q_j} \le k_j \le \frac{A+t_{j-1}}{q_j}
\end{equation}
Let $t_j = k_{j}p_{j}$. It is obvious that $t_{j} \in \N$. We use the fact $\xi_{j} = p_{j}/q_{j}$ to see that $t_{j} / \xi_{j} = k_{j}p_{j} / \xi_{j} = k_{j}q_{j}$. Hence $t_{j} / \xi_{j} \in \N$.

From \eqref{eqn:double:inequality}, we have $1 +t_{j-1} \le k_{j}q_{j} \le A +t_{j-1}$. Hence $1 \le k_{j}q_{j} -t_{j-1} \le A $. Using $t_{j} / \xi_{j} = k_{j}q_{j}$, we immediately have $1 \le t_{j} /\xi_{j} -t_{j-1} \le A$, verifying (\ref{item:item:sequence:ApproxUpper}).

We next verify that the sequence $\{t_{j}\}$ in (\ref{item:item:sequence:ApproxUpper}) satisfies (\ref{item:item:sequence exponentialgrowth}). From (\ref{item:item:sequence:ApproxUpper}), we have $0 \le t_{j}/\xi_{j} - t_{j-1}$. Thus
\begin{equation}
    \tag{$\diamondsuit$}
\label{eqn:SequenceIncreases}
t_{j} \ge \xi_{j}\,t_{j-1} \qquad \text{for all $j>1$.}
\end{equation}
For any $j >m$, we apply \eqref{eqn:SequenceIncreases} iteratively $m$ times to get 
\begin{align*}
    t_{j} &\ge \xi_{j} \, t_{j-1}\\
    &\ge \xi_{j}\,\xi_{j-1}\,t_{j-2} \ge \cdots \\
    &\ge \Bigl (\xi_{j}\,\xi_{j-1} \cdots\, \xi_{j -m +1} \Bigr) t_{j-m}\\
    & = w(\gamma)t_{j-m}
\end{align*}
Further applying $t_{j} \ge w(\gamma)t_{j-m}$  iteratively $k$ times (and using that $w(\gamma) \ge 1$) gives $t_j \ge w(\gamma)^{n}t_{j-nm}$ for any $n \ge 1$. In particular, with $j = nm +1$ we have $t_{nm+1} \ge w(\gamma)^{n}t_{1}$ which confirms (\ref{item:item:sequence exponentialgrowth}).

In order to establish (\ref{item:item:sequenceexponentialbelow}), we recall that $w(\gamma) >1$. It follows that the spirality of $S$ is non-trivial. Let $\epsilon$ be the governor of $g$ with respect to the chosen mapping torus (see Definition~\ref{defn:DilationSlopes}). We note that $\epsilon$ is strictly greater than $1$ since the spirality of $S$ is non-trivial. We will show by induction on $j$ that 
\[
t_{j} \le (A+t_1)\sum_{i=1}^{j}\epsilon^{i}
\]
When $j=1$, it is obvious that $t_1 < (A+t_1)\epsilon$ (using $\epsilon >1$). For the inductive step, we use the right inequality (i.e, $t_{j}/\xi_{j} -t_{j-1} \le A$) in (\ref{item:item:sequence:ApproxUpper}), and the fact $\xi_{j} \le \epsilon$ to get that
\begin{align*}
    t_{j} &\le (A +t_{j-1})\,\xi_{j}\\
    &\le (A +t_{j-1})\epsilon = A\epsilon + t_{j-1}\epsilon\\
    &\le A\epsilon + (A+t_1) \Bigl (\sum_{i=1}^{j-1}\epsilon^{i} \Bigr )\epsilon = A\epsilon + (A+t_1)\sum_{i=1}^{j-1}\epsilon^{i+1}\\
    &< (A+t_1)\epsilon + (A+t_1)\sum_{i=1}^{j-1}\epsilon^{i+1}\\
    &=(A+ t_1) \Bigl (\epsilon + \sum_{i=1}^{j-1}\epsilon^{i+1} \Bigr ) = (A +t_1) \sum_{i=1}^{j}\epsilon^{i}
    \end{align*}
Since $\sum_{i=1}^{j}\epsilon^{i} < \epsilon^{j+1} / (\epsilon-1)$, we obtain $t_{j} < (A+t_1) \epsilon^{j+1} / (\epsilon-1)$. It follows that there is $D>0$ depending only on $A$, $\epsilon$ and $t_1$ such that $t_j \le e^{Dj +D}$ for all $j \ge 1$ (for example, we may choose $D = \ln(\epsilon) + \ln (A+t_1)\epsilon / (\epsilon -1)$).
\end{proof}

For the rest of this section, we fix the curve $\gamma$ satisfying the hypothesis of Lemma~\ref{lem:construct sequence}. The collection $\mathcal{T}_{g}$ subdivides $\gamma$ into a concatenation $\gamma_1 \cdots \gamma_m$ with the following properties. Each path $\gamma_{i}$ belongs to a piece $B_{i}$ of $S$, starting on a circle $c_{i} \in \mathcal{T}_{g}$ and ending on the circle $c_{i+1}$. The image $g(\gamma_{i})$ of this path in $N$ lies in a block $M_i$. The image of the circle $g(c_{i})$ in $N$ lies a JSJ torus $T_{i}$ obtained by gluing to a boundary torus $\overleftarrow{T_i}$ of $M_{i-1}$ to a boundary torus $\overrightarrow{T_i}$ of $M_{i}$.  We extend the sequence $\gamma_{1},\cdots,\gamma_{m}$ to a periodic sequence $\{\gamma_{j}\}_{j=1}^{\infty}$ with $\gamma_{j+m} = \gamma_{j}$ for all $j \ge 1$. We extend the sequence $c_1, \cdots, c_m$ to a periodic sequence $\{c_j\}_{j=1}^{\infty}$ with $c_{j+m} =c_{j}$ for all $j \ge 1$.  We also choose the basepoint $x_{0} \in S$ to be the initial point of $\gamma_1$. Let $\mathcal{Q}$ be the family of lines that are lifts of loops of $\mathcal{T}_{g}$.

\begin{cons}[Constructing a sequence of points in $\tilde{S}$]
\label{cons:sequence x}
We recall that $\tilde{g} \colon \tilde{S} \to \tilde{N}$ is an embedding. Let $\tilde{x}_{0} = \tilde{g}(\tilde{s}_{0})$.
For convenience, we relabel $\tilde{x}_0$ by $z_{1}$. Draw a line that passes through $z_{1}$ parallel to a lift of the degeneracy slope $\overrightarrow{\mathbf{s}_1} = \mathbf{s}_{c_{1}B_{1}}$ in $\tilde{M}_1 = \tilde{M}_{B_{1}}$. We denote the intersection of this line with the slice $\tilde{B}_1 \times \{t_{1}\}$ (of $\tilde{M}_1 = \tilde{M}_{B_{1}}$) by $y_{1}$. We denote the degeneracy slope $\mathbf{s}_{c_1B_m}$ by $\overleftarrow{\mathbf{s}_1}$. Let $x_1 = z_1$. We construct a sequence of triples $\{x_j,y_j,z_j\}$ inductively as the following.

Suppose that $y_{j-1}$, $x_{j-1}$, and $z_{j-1}$ have been defined. Let $\tilde{\gamma}_{j-1}$ be the lift of $\gamma_{j-1}$ in $\tilde{N}$ based at $y_{j-1}$. Let $y'_{j}$ be the terminal point of $\tilde{\gamma}_{j-1}$. Draw a line that passes through $y'_j$ parallel to a lift of the degeneracy slope $\overleftarrow{\mathbf s_{j}} = \mathbf{s}_{c_{j}B_{j-1}}$ in $\tilde{M}_{j-1} = \tilde{M}_{B_{j-1}}$. This line meets the slice $\tilde{B}_{j-1} \times \{0\} \subset \tilde{S}$ at a point denoted by $x_{j}$, and it meets the slice $\tilde{B}_{j-1} \times \{t_{j} / \xi_{j} \}$ in $\tilde{M}_{j-1} = \tilde{M}_{B_{j-1}}$ at a point denoted by $y_{j}$.  Draw a line that passes through $y_{j}$ parallel to a lift of the degeneracy slope $\overrightarrow{\mathbf s_{j}} = \mathbf{s}_{c_{j}B_{j}}$ in $\tilde{M}_{j} = \tilde{M}_{B_{j}}$ until it meets the slice $\tilde{B}_{j} \times \{0\} \subset \tilde{S}$ at a point denoted by $z_{j}$.
\end{cons}

\begin{figure}[h]
\centering
\def\svgwidth{\columnwidth}
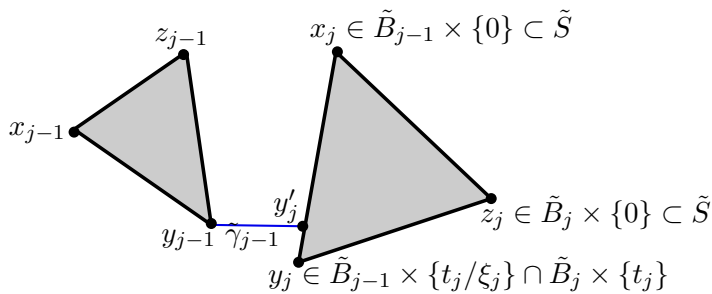
\caption{The picture illustrates the slices that $x_j$, $y_j$, $z_j$ belong to}
\label{figure:2}
\end{figure}

In the following, let $\{x_{j}\}$, $\{y_{j}\}$, and $\{z_{j}\}$ be the collections of points given by Construction~\ref{cons:sequence x}. We note that $z_{j} \in (\tilde{B}_{j-1} \times \{0\}) \cap (\tilde{B}_{j} \times \{0\})$.
We denote
$\overleftarrow{\lambda_{i}}$ the length of the image of $\overleftarrow{\mathbf s_i}$ in $N$, and $\overrightarrow{\lambda_{i}}$ the length of the image of $\overrightarrow{\mathbf s_i}$ in $N$ with $i=1,\dots,m$.
We extend the sequence $\overleftarrow{\lambda_{1}}, \dots, \overleftarrow{\lambda_{m}}$ to the $m$--periodic sequence $\{\overleftarrow{\lambda_{j}}\}_{j=1}^{\infty}$, and the sequence $\overrightarrow{\lambda_{1}}, \dots, \overrightarrow{\lambda_{m}}$ to the $m$--periodic sequence $\{\overrightarrow{\lambda_{j}}\}_{j=1}^{\infty}$.

\begin{rem}
\label{rem:distance}
From Construction~\ref{cons:sequence x}, it is possible that $x_{j} = z_{j}$. For each $j \ge 2$, since $y_{j}$ belongs to the slice $\tilde{B}_{j-1} \times \{t_j / \xi_j\}$ of $\tilde{M}_{j-1}$, we have
$d(y_j,x_j) = \frac{t_j}{\xi_j} \overleftarrow{\lambda_j}$.
Also $y_j$ belongs to the slice $\tilde{B}_{j} \times \{t_j\}$ of $\tilde{M}_{j}$ (see Lemma~\ref{lem: application rule of sines} for a similar argument), and thus
$d(y_j,z_j) = t_{j}\overrightarrow{\lambda_j}$.

Since $y'_{j} \in \tilde{B}_{j-1} \times \{t_{j-1} \}$, it follows that $d(x_j, y'_j) = t_{j-1}\overleftarrow{\lambda_j}$. Thus
$d(y_j,y'_{j}) = d(y_j,x_j) -d(x_j, y'_j)  = \frac{t_j}{\xi_j}\overleftarrow{\lambda_j} - t_{j-1}\overleftarrow{\lambda_j} = (\frac{t_j}{\xi_j} - t_{j-1})\overleftarrow{\lambda_j}$
\end{rem}

We use the following lemma in the proof of Lemma~\ref{lem:doubleexponentiallower} where we show that the double spiral path $\sigma_n$ in Definition~\ref{defn:doublespiralpath} is non-backtracking.
\begin{lem}
\label{lem:constantbacktracking}
Let $F$ be a connected compact hyperbolic surface with non-empty boundary. Equip $F$ with a hyperbolic metric and lift this metric to the universal cover $\tilde{F}$, and  denote it by $d$. There exists a constant $k >0$ such that the following holds:

Let $\ell$     and $\ell'$ be two distinct boundary lines in $\tilde{F}$. Let $[p,p']$ be a geodesic of shortest length from $\ell$ to $\ell'$. Let $c_1$ be the boundary circle in $F$ such that $\ell'$ covers $c_1$. Let $z$ be a point in $\ell$, $\beta$ be the geodesic in $\tilde{F}$ connecting $z$ to $p'$, and $\tilde{c}_1^{k}$ be the path lift of $c_{1}^{k}$ based at $p'$. Let  $\beta^{-1}$ be the path lift of the image of the inverse path $\bar{\beta}$ (of $\beta$) based at $\tilde{c}_1^{k}(1)$. Then the terminal point of the path $\beta \cdot \tilde{c}_{1}^{k} \cdot \beta^{-1}$ does not lie in the boundary line $\ell$.
\end{lem}
\begin{proof}
We note that $(\tilde{F},d)$ is bilipschitz homeomorphic to a fattened tree (see the paragraph after Lemma~1.1 \cite{NeuBeh08}). Thus, there exists $\epsilon >0$ such that the following holds: Let $\ell_0$   and $\ell_1$ be two distinct boundary lines in $\tilde{F}$. Let $[p_0,p_1]$ be a geodesic of shortest length from $\ell_0$ to $\ell_1$. Then every geodesic from $\ell_0$ to $\ell_1$ must come within a distance $\epsilon$ of both $p_0$ and $p_1$.
Let $\delta$ be the minimum of values $\abs{c}$ with $c$ varying over all boundary circles of $F$. Let $k$ be a sufficiently large integer such that $k \delta > 2 \epsilon$.

Suppose by way of contradiction that the terminal point of  $\beta \cdot \tilde{c}_{1}^{k} \cdot \beta^{-1}$ lies in $\ell$. Let $h = [c_{1}^k] \in \pi_1(F,*)$ where $* \in c_1$ is the image of $p'$ under the covering map. We note that $\beta^{-1} = h(\bar{\beta})$, $h(p') \in \ell'$, and $p,z \in \ell$.
Since we assume that $\beta^{-1}(1) \in \ell$, it follows that $h(z) = \beta^{-1}(1) \in \ell$. Since $z, p \in \ell$, it follows that $h(p) \in \ell$. Let $[h(p), h(p')]$ be the geodesic in $\tilde{F}$ connecting $h(p) \in \ell$ to $h(p') \in \ell'$.  According to the previous paragraph, it follows that there exists $a \in [h(p), h(p')]$ such that $d(p', a) \le \epsilon$. Since the concatenation $[h(p),a] \cdot [a,p']$ is a path from $\ell$ to $\ell'$ and $[p,p']$ is a geodesic of shortest length from $\ell$ to $\ell'$, it follows that $d(p,p') = \bigabs{[p,p']} \le \bigabs{[h(p),a] \cdot [a,p']} = d(h(p),a) + d(a,p') \le d(h(p),a) + \epsilon$. 

Using the fact $d(h(p), h(p')) = d(p,p')$, $a \in [h(p),h(p')]$, and the inequality 
$d(p,p') \le d(h(p),a) + \epsilon$, we have 
\[
d(a, h(p')) = d(h(p),h(p')) - d(a, h(p)) = d(p,p') -d(a, h(p)) \le \epsilon
\]
It follows that
\[
k\delta \le  k\abs{c_1} = \abs{\tilde{c}_{1}^{k}} = d(p', h(p')) \le d(p',a) + d(a, h(p')) \le \epsilon + d(a, h(p')) \le 2\epsilon
\]
This contradicts to our choice of $k$ that $k\delta > 2\epsilon$.
\end{proof}
We recall that $\mathcal{Q}$ is the family of lines in $\tilde{S}$ that are lifts of curves of $\mathcal{T}_g$. We also recall that both $z_j$ and $z_{j+1}$ belong $\tilde{B}_j$. For each $j$, let $\ell_j$ be the line in $\mathcal{Q}$ that passes through $z_j$. Let $\alpha_j$ be a shortest path in $\tilde{B}_j$ connecting $\ell_j$ to $\ell_{j+1}$ of $\tilde{B}_j$.
For each piece $B$ of $S$, let $k_B$ be the constant given by Lemma~\ref{lem:constantbacktracking}. Let $k_0$ be the maximum of $k_B$ where $B$ runs over pieces of $S$. 

\begin{defn}[Double spiral path]
\label{defn:doublespiralpath}
Let $[z_j, z_{j+1}]$ be a geodesic in $\tilde{S}$ connecting $z_j$ to $z_{j+1}$, and $[z_j, \alpha_{j}(1)]$ a geodesic in $\tilde{S}$ connecting $z_j$ to $\alpha_{j}(1)$. For each $n \ge 1$, let $\tau_{n}$ be the concatenation of the geodesics 
\[
[z_{1},z_{2}],[z_{2},z_{3}],\dots,[z_{nm-1},z_{nm}],[z_{nm},\alpha_{nm}(1)]
\]
Let $\widetilde{c}_{1}^{k_0}$ be the path lift of $c_{1}^{k_0}$ based at $\alpha_{nm}(1)$. We define \emph{double spiral path} $\sigma_n$ of $\tau_n$ as
\[
\sigma_{n} = \tau_{n}\cdot \widetilde{c}_{1}^{k_0} \cdot \tau^{-1}_{n}
\]
\end{defn}


\begin{lem}
\label{lem:lowerdouble}
Suppose that $M_1$ is a hyperbolic block of $N$. Then there exists an integer $n_1$, a function $F(n)$ such that $F(n) \sim e^{e^{n}}$, and $F(n) \le \abs{\alpha_{nm+1}}_{\tilde{S}}$ for all $n \ge n_1$.
\end{lem}
\begin{proof}
We recall that $M_{B_1}$ is the mapping torus of a pseudo-Anosov $\varphi \colon B_1 \to B_1$. By our assumption, we note that $M_{nm+1} =M_1$ and $M_1$ is a hyperbolic block of $N$.
Applying Lemma~\ref{lem:pseudo} to the pseudo-Anosov $\varphi$ and the curve $\gamma_1$, there exists a sufficiently large integer $n_0$ such that the following holds: For any $j \ge n_0$, let $u_j$ and $v_j$ be the two endpoints of a path lift of $\varphi^{j}(\gamma_1)$ in the universal cover $\tilde{B}_1$. Let $\alpha$ be a shortest path in $\tilde{B}_1$ from a boundary line of $\partial \tilde{B}_1$ that contains $u_j$ to a boundary line of $\partial \tilde{B}_1$ that contains $v_j$. Then
\begin{equation}
\tag{$*$}
\label{eq:star}
e^{Lj}/C \le d_{\tilde{S}} \bigl (u_j,v_j \bigr ), \,  \frac{\ln d_{\tilde{S}} \bigl (u_j,\alpha(0) \bigr )}{n} \le \frac{1}{2},\, \textup{and}\,\,\, \frac{\ln d_{\tilde{S}} \bigl (v_{j},\alpha(1) \bigr )}{n} \le \frac{1}{2}
\end{equation}
By (\ref{item:item:sequence exponentialgrowth}) in Lemma~\ref{lem:construct sequence}, we have
$ \bigl ( w(\gamma) \bigr )^{n}t_{1} \le t_{nm+1}$. Hence, there exists an integer $n_{1}$ which is greater than $n_0$ and satisfies that $n_0 \le \bigl ( w(\gamma) \bigr )^{n}t_{1} \le t_{nm+1}$ for all $n \ge n_1$. 
Let $F$ be a function on $n$ defined by $F(n) = e^{Lt_{1}w(\gamma)^{n}}/  C - 2e^{n/2}$. We note that $w(\gamma) >1$, thus $F(n) \sim e^{e^n}$ (we recall that if $a, b>1$ then $a^{b^{n}} \sim e^{e^{n}}$).
To prove the lemma, we first prove the following claim.

{\bf{Claim:}} $F(n) \le d_{\tilde{S}} \bigl (z_{nm+1},\alpha_{nm+1}(1) \bigr ) -e^{n/2}$ for all $n \ge n_1$.  

We recall that $\alpha_{nm+1}$ is a shortest path connecting $\ell_{nm+1}$ to $\ell_{nm+2}$. Let $\widetilde{\varphi^{t_{nm+1}}(\gamma_1)}$ be the path lift of $\varphi^{t_{nm+1}}(\gamma_1)$ based at $z_{nm+1}$. We denote the initial point and terminal point of $\widetilde{\varphi^{t_{nm+1}}(\gamma_1)}$ by $u_{t_{nm+1}}$ and $u_{t_{nm+1}}$ respectively. We have that $v_{nm+1} \in \ell_{nm+2}$ (we refer the reader to Lemma~\ref{lem:refereeask} for an explanation of this fact).  Using (\ref{eq:star}), the fact $z_{nm+1} =u_{t_{nm+1}}$, and the triangle inequality, we have
\begin{align*}
    F(n) &= e^{Lt_{1}w(\gamma)^{n}}/  C - 2e^{n/2} \le e^{Lt_{nm+1}}/C -2e^{n/2}\\
    &\le d_{\tilde{S}}(u_{t_{nm+1}}, v_{t_{nm+1}}) -2e^{n/2}  \,\,\, \textup{(using the first inequality of (\ref{eq:star}))}\\
    &= d_{\tilde{S}}(z_{nm+1}, v_{t_{nm+1}}) -2e^{n/2}\\
    & \le d_{\tilde{S}}(z_{nm+1}, \alpha_{nm+1}(1)) + d_{\tilde{S}}(\alpha_{nm+1}(1), v_{t_{nm+1}}) -2e^{n/2}\\
    &\le d_{\tilde{S}}(z_{nm+1}, \alpha_{nm+1}(1)) + e^{n/2} -2e^{n/2} \,\,\, \textup{(using the third inequality of (\ref{eq:star}))}\\
    &= d_{\tilde{S}}(z_{nm+1}, \alpha_{nm+1}(1)) -e^{n/2} = d_{\tilde{S}}(u_{t_{nm+1}}, \alpha_{nm+1}(1)) -e^{n/2}\\
    &\le d_{\tilde{S}}(u_{t_{nm+1}}, \alpha_{nm+1}(0)) + d_{\tilde{S}}(\alpha_{nm+1}(0), \alpha_{nm+1}(1)) - e^{n/2}\\
    &\le e^{n/2} + d_{\tilde{S}}(\alpha_{nm+1}(0), \alpha_{nm+1}(1)) - e^{n/2} \,\,\, \textup{(using the second inequality of (\ref{eq:star}))}\\
    &= d_{\tilde{S}}(\alpha_{nm+1}(0), \alpha_{nm+1}(1)) = \abs{\alpha_{nm+1}}_{\tilde{S}}
\end{align*}
which is confirming the lemma.
\end{proof}

In the proof of Lemma~\ref{lem:lowerdouble}, we claim that the terminal point of $\widetilde{\varphi^{t_{nm+1}}(\gamma_1)}$ lies in $\ell_{nm+2}$ without an explanation. We provide a proof for this claim in the following lemma.
\begin{lem}
\label{lem:refereeask}
Let $\widetilde{\varphi^{t_{nm+1}}(\gamma_1)}$ be the path lift of $\varphi^{t_{nm+1}}(\gamma_1)$ based at $z_{nm+1}$. Then the terminal point of $\widetilde{\varphi^{t_{nm+1}}(\gamma_1)}$ lies in $\ell_{nm+2}$ 
\end{lem}
\begin{proof}
In the mapping torus $M_{B_1}$, we note that $\mathbf{s}_{c_1 B_1} \cdot \varphi^{j}(\gamma_1) \cdot \mathbf{s}^{-1}_{c_2 B_1} \simeq \varphi^{j+1}(\gamma_1)$. We will show by induction  on $k \ge 1$ that  $\mathbf{s}^{k}_{c_1 B_1} \cdot \gamma_1 \cdot \mathbf{s}^{-k}_{c_2 B_1} \simeq \varphi^{k}(\gamma_1)$. The base case $k =1$ is obvious (we let $j =0$). For inductive step, assume it is true for $k =n-1$. To prove for $k =n$, we let $j =n-1$ to see that
$\mathbf{s}_{c_1 B_1}^{n} \cdot \gamma_1 \cdot \mathbf{s}_{c_1 B_1}^{-n} \simeq \mathbf{s}_{c_1 B_1} \cdot \bigl ( \mathbf{s}_{c_1 B_1}^{n-1} \cdot \gamma_1 \cdot \mathbf{s}_{c_2 B_1}^{1-n} \bigr ) \cdot \mathbf{s}^{-1}_{c_2 B_1} \simeq \mathbf{s}_{c_1 B_1} \cdot \varphi^{n-1}(\gamma_1) \cdot \mathbf{s}^{-1}_{c_2 B_1} \simeq \varphi^{n}(\gamma_1)$. The claim is confirmed. 

Let $k = t_{nm+1}$. By lifting property, the terminal point of $\widetilde{\varphi^{t_{nm+1}}(\gamma_1)}$ is the terminal point of the path lift of $\mathbf{s}^{t_{nm+1}}_{c_1 B_1} \cdot \gamma_1 \cdot \mathbf{s}^{-t_{nm+1}}_{c_2 B_1}$ based at $z_{nm+1} = \widetilde{\varphi^{t_{nm+1}}(\gamma_1)}(0)$. Therefore, the terminal point of $\widetilde{\varphi^{t_{nm+1}}(\gamma_1)}$ lies in $\ell_{nm+2}$.
\end{proof}
\begin{lem}
\label{lem:doubleexponentiallower}
Suppose that $M_{1}$ is a hyperbolic block of $N$.
Then $e^{e^{n}}$ is dominated by $d_{\tilde{S}}(z_1, \sigma_{n}(1)$.
\end{lem}
\begin{proof}
Let $F$ be the function given by Lemma~\ref{lem:lowerdouble}. We recall that $F(n) \sim e^{e^n}$. Let $n_1$ be the constant given by Lemma~\ref{lem:lowerdouble}. For each $n \ge n_1 +1$, let $\bigl [z_{1},\sigma_{n}(1) \bigr ]$ be the geodesic in $\tilde{S}$ connecting $z_{1}$ to $\sigma_{n}(1)$. Let $\mathbf{T}_S$ be the tree given by Definition~\ref{defn:map of trees}. The subpath $[z_{j},z_{j+1}]$ (with $j = 1, \cdots nm-1$), the subpath $[z_{nm}, \alpha_{nm}(1)] \cdot \tilde{c}_{1}^{k_0} \cdot [z_{nm}, \alpha_{nm}(1)]^{-1}$, and the subpath $[z_{j},z_{j+1}]^{-1}$ (with $j = 1, \cdots nm-1$) of $\sigma_n$ belong to pieces of $\tilde{S}$. These pieces  correspond to vertices of $\mathbf{T}_S$. By our construction of $\sigma_n$ and by Lemma~\ref{lem:constantbacktracking}, these vertices are distinct. It follows that the geodesic $[z_1,\sigma_{n}(1)]$ must come through all these pieces. Thus, $\bigl [z_{1},\sigma_{n}(1) \bigr ]$ must come through the piece $\tilde{B}_{(n-1)m+1}$, and $\bigl [z_{1},\sigma_{n}(1) \bigr ]$ enters $\tilde{B}_{(n-1)m+1}$ at $\ell_{(n-1)m}$ and leaves $\tilde{B}_{(n-1)m+1}$ at $\ell_{(n-1)m+1}$. Using the fact that $\alpha_{(n-1)m+1}$ is a shortest path from $\ell_{(n-1)m}$ to $\ell_{(n-1)m+1}$. It follows that $\bigl |\alpha_{(n-1)m+1} \bigr |_{\tilde{S}} \le d_{\tilde{S}}(z_1, \sigma_{n}(1))$. As $F(n) \sim e^{e^{n}}$, it implies that $F(n-1) \sim e^{e^{n}}$. Using Lemma~\ref{lem:lowerdouble}, we have $F(n-1) \le \bigl |\alpha_{(n-1)m+1} \bigr |_{\tilde{S}} \le d_{\tilde{S}}(z_1, \sigma_{n}(1))$. The lemma is verified.
\end{proof}
The proof of the following lemma is similar to the proof of Lemma~5.4 in \cite{Hruska-Nguyen}.
\begin{lem}
\label{lem:linearbound}
The distance in $\tilde{N}$ between the endpoints of $\sigma_{n}$ is bounded above by a linear function of $n$.
\end{lem}
\begin{proof}
We recall that $\tau_{n}$ is the concatenation of geodesics 
\[
[z_{1},z_{2}],[z_{2},z_{3}],\dots,[z_{nm-1},z_{nm}],[z_{nm},\alpha_{nm+1}(1)]
\]
Note that in the Construction~\ref{cons:sequence x}, we also produce points $y_1,\dots,y_{nm+1}$ in $\tilde{N}$. Similarly, we also have points $\bar{y}_{nm+1},\dots,\bar{y}_{1}$ associating to $\tau^{-1}_{n}$.
Our purpose is to show the distance in $(\tilde{N},d)$ between the endpoints of $\sigma_{n}$ is bounded above by a linear function of $n$.
By the triangle inequality it suffices to produce an upper bound for the distance between successive points of the linear sequence 
\[
y_{1},y_{2},\dots,y_{nm+1}, \bar{y}_{nm+1},\dots, \bar{y}_{1}
\]
Let $A$ is the constant given by Lemma~\ref{lem:construct sequence}.
Let $\bar{A}$ be the maximum of all possible numbers
$\max \bigl \{\bigabs{\gamma_{j}} + A\overleftarrow{\lambda_{j+1}} +k_{0} |c_{j}| + t_{1}\overrightarrow{\lambda_{1}} \bigr \}$.
By Remark~\ref{rem:distance} and Lemma~\ref{lem:construct sequence} we have 
$d(y'_{j},y_{j}) = \bigl (t_{j} / \xi_{j} - t_{j-1} \bigr )\overleftarrow{\lambda_{j}} \le A\overleftarrow{\lambda_{j}}$.
Using the triangle inequality, we have
$d(y_{j},y_{j+1}) \le d(y_{j},y'_{j+1}) + d(y'_{j+1},y_{j+1}) \le \bigabs{\gamma_{j}} + d(y'_{j+1},y_{j+1}) \le \bigabs{\gamma_{j}} + A\overleftarrow{\lambda_{j+1}} \le \bar{A}$
for all $j \ge 0$. Therefore $d(y_{1},y_{mn+1}) \le \bar{A}mn$.
Similarly, we have $d(\bar{y}_{1},\bar{y}_{mn+1}) \le \bar{A}mn$. We note that two points $y_{nm+1}$ and $\bar{y}_{nm+1}$ belong to the same plane $\tilde{T}_{nm+1}$ and 
$d(y_{nm+1},\bar{y}_{nm+1}) \le k_{0}|c_{1}| \le \bar{A}$.
Thus, 
$d(y_{1},\bar{y}_{1}) \le d(y_1, y_{nm+1}) + d(y_{nm+1},\bar{y}_{nm+1}) + d(\bar{y}_{nm+1},\bar{y}_1) \le 2\bar{A}mn + \bar{A}$.
Since $d(z_{1},y_{1}) = t_{1}\overrightarrow{\lambda_{1}} \le \bar{A}$ and $d(\bar{z}_{1},\bar{y}_{1}) = t_{1}\overrightarrow{\lambda_{1}} \le \bar{A}$, it follows that $d(\sigma_{n}(0), \sigma_{n}(1))= d(z_1, \bar{z}_1) \le d(z_1,y_1) + d(y_1,\bar{y}_1) + d(\bar{y}_1, \bar{z}_1) \le 2\bar{A}nm +3\bar{A}$.
\end{proof}

\begin{proof}[Proof of Proposition~\ref{prop:lower}]
If there is a closed curve $\gamma$ satisfying the hypothesis of Lemma~\ref{lem:construct sequence} and passing through a hyperbolic block, then Proposition~\ref{prop:lower} is confirmed by a combination of the previous lemmas in this section. What remains to be shown is that the existence of the curve $\gamma$.

Since the spirality of $S$ is non-trivial, we can choose a closed curve $\alpha$ in $S$ with nonempty intersection with $\mathcal{T}_{g}$ such that $w(\alpha) >1$.
If the curve $\alpha$ already passes through a geometrically infinite piece, then we let $\gamma = \alpha$.
If not, we need to extend the curve $\alpha$ to a new closed curve $\gamma$ so that $w(\gamma) = w(\alpha)$ and $\gamma$ has nonempty intersection with a curve in $\mathcal{T}_{g}$ that is a boundary component of a piece of $S$, and this piece  is mapped into a hyperbolic block of $N$. We describe below how we find such a curve $\gamma$.

The collection $\mathcal{T}_{g}$ subdivides $\alpha$ into a concatenation $\alpha_{1}\cdots \alpha_{n}$ such that each $\alpha_{i}$ belongs to a piece $B_{i}$ of $S$, starting on a circle $c_{i} \in \mathcal{T}_{g}$ and ending on the circle $c_{i+1}$. Note that by our assumption above, each piece $B_i$ is horizontal surface.
We recall that $\Gamma(\mathcal{T}_{g})$ is the graph dual to the collection $\mathcal{T}_{g}$ on $S$. Let $v_{i}$ be the vertex in $\Gamma(\mathcal{T}_{g})$ associated to the piece $B_{i}$.
The closed curve $\alpha$ determines the closed cycle $e_{1}\cdots e_{n}$ in $\Gamma(\mathcal{T}_{g})$ where the initial vertex and terminal vertex of the edge $e_{i}$ are $v_{i}$ and $v_{i+1}$ respectively (with a convention that $v_{n+1} =v_{1})$. 

Since $S$ contains a geometrically infinite piece, let $u$ be the vertex in $\Gamma(\mathcal{T}_{g})$ associated to this piece. It follows that $u \neq v_{i}$ for any $i =1,\dots,n$. There exists $j \in \{1, \dots, n\}$ such that the following holds: there exists a $\beta$ in $\Gamma(\mathcal{T}_{g})$ with no self intersection connecting $v_j$ to $u$ such that the other vertices $v_{i}$ with $i \neq j$ does not appear on $\beta$. Without loss generality, we assume $j =1$.
 Note that $S$ is a clean almost fiber surface, so every piece of $S$ is neither an annulus or a disk. We thus choose a path $\gamma'$ connecting $\alpha_{1}(0)$ to $\alpha_{1}(1)$ with non-empty intersection  with $\mathcal{T}_{g}$ such that it can not be homotoped out of the geometrically finite piece, and the corresponding path of $\gamma' \subset S$ in $\Gamma(\mathcal{T}_{g})$ is the back-tracking path $\beta \cdot \beta^{-1}$. Let $\gamma$ be the concatenation of $\gamma'\cdot \alpha_{2}\cdots \alpha_{n}$. It follows that $w(\gamma) = w(\alpha)$. Since $w(\alpha) >1$, we obtain $w(\gamma) >1$.
\end{proof}

\section{Distortion of surfaces in non-geometric 3-manifolds}
\label{section:putting results}

In Section~\ref{sec:almostfiber}, we show that the distortion of a clean surface subgroup in a non-geometric $3$--manifold group can be determined by looking at the distortion of the clean almost fiber part. We recall that the almost fiber part contains only horizontal and geometrically infinite pieces.
The distortion of properly immersed $\pi_1$--injective horizontal surfaces in graph manifolds is computed in \cite{Hruska-Nguyen}. In the setting of mixed manifold, the distortion of a clean almost fiber part  is addressed in Section~\ref{sec:distortion mixed}. 
In this section, we compute the distortion of arbitrary clean surface in a non-geometric $3$--manifold by putting the previous results together.

\begin{lem}
\label{lem:distortioningraphmanifold}
Let $S$ be a clean almost fiber surface in a  graph manifold $N$. Let $\Delta$ be the distortion of $\pi_1(S)$ in $\pi_1(N)$. If $S$ contains only one horizontal piece then $\Delta$ is linear. If $S$ contains at least two horizontal pieces, then $\Delta$ is quadratic if the spirality of $S$ is trivial, otherwise it is exponential.
\end{lem}
\begin{proof}
The fundamental group of a Seifert fibered block in $\pi_1(N)$ is undistorted. If $S$ contains only one horizontal piece then $\Delta$ is linear by Remark~\ref{rem:distortionwellknow} and Proposition~\ref{prop:distortion}.
We now consider the case $S$ has at least two horizontal pieces. We remark that the main theorem in \cite{Hruska-Nguyen} states for properly immersed $\pi_1$--injective, horizontal surfaces. However, the proof of the main theorem in \cite{Hruska-Nguyen} still hold for clean almost fiber surfaces.
\end{proof}

\begin{proof}[Proof of Theorem~\ref{thm:intro:graphmanifold}]
The proof is a combination of Lemma~\ref{lem:distortioningraphmanifold}, Theorem~\ref{thm:distortionalmostfiber}, and (1) in Remark~\ref{rem:SUNLIU}.
\end{proof}


\begin{proof}[Proof of Theorem~\ref{thm:thesismain}]
If $\Phi(S)$ is empty then the distortion of $\pi_1(S)$ in $\pi_1(N)$ is linear by Corollary~\ref{cor:linear distortion}. We now assume that $\Phi(S)$ is non-empty. By Theorem~\ref{thm:distortionalmostfiber}, it is suffice to compute the distortion of each component of the almost fiber part $\Phi(S)$ in $N$.

Let $N'$ be a submanifold of $N$ such that the restriction $g|_{S'} \colon S' \looparrowright N'$ is a clean almost fiber surface. Note that $\pi_1(N')$ is undistorted in $\pi_1(N)$, thus the distortion of $\pi_1(S')$ in $\pi_1(N)$ is equivalent to the distortion of $\pi_1(S')$ in $\pi_1(N')$. To compute the distortion of $\pi_1(S')$ in $\pi_1(N')$, we note that the distortions of clean almost fiber surfaces in mixed manifolds, graph manifolds, Seifert fibered spaces and hyperbolic spaces are addressed in Theorem~\ref{thm:introductionUpperBound}, Theorem~\ref{thm:intro:graphmanifold} and Remark~\ref{rem:distortionwellknow} respectively. The proof of the theorem follows easily by combining these results together with (1) in Remark~\ref{rem:SUNLIU}.
\end{proof}

\bibliographystyle{alpha}
\bibliography{hoang}
\end{document}

%% file: drawing-7.pdf_tex
\begingroup%
  \makeatletter%
  \providecommand\color[2][]{%
    \errmessage{(Inkscape) Color is used for the text in Inkscape, but the package 'color.sty' is not loaded}%
    \renewcommand\color[2][]{}%
  }%
  \providecommand\transparent[1]{%
    \errmessage{(Inkscape) Transparency is used (non-zero) for the text in Inkscape, but the package 'transparent.sty' is not loaded}%
    \renewcommand\transparent[1]{}%
  }%
  \providecommand\rotatebox[2]{#2}%
  \newcommand*\fsize{\dimexpr\f@size pt\relax}%
  \newcommand*\lineheight[1]{\fontsize{\fsize}{#1\fsize}\selectfont}%
  \ifx\svgwidth\undefined%
    \setlength{\unitlength}{550.58998658bp}%
    \ifx\svgscale\undefined%
      \relax%
    \else%
      \setlength{\unitlength}{\unitlength * \real{\svgscale}}%
    \fi%
  \else%
    \setlength{\unitlength}{\svgwidth}%
  \fi%
  \global\let\svgwidth\undefined%
  \global\let\svgscale\undefined%
  \makeatother%
  \begin{picture}(1,0.52372394)%
    \lineheight{1}%
    \setlength\tabcolsep{0pt}%
    \put(0,0){\includegraphics[width=\unitlength,page=1]{drawing-7.pdf}}%
    \put(0.1256624,0.43628307){\color[rgb]{0,0,0}\makebox(0,0)[lt]{\lineheight{1.25}\smash{\begin{tabular}[t]{l}$\beta_{j_j}$\end{tabular}}}}%
    \put(0.33540682,0.50900851){\color[rgb]{0,0,0}\makebox(0,0)[lt]{\lineheight{1.25}\smash{\begin{tabular}[t]{l}$\sigma_{s}$\end{tabular}}}}%
    \put(0.67247697,0.44104031){\color[rgb]{0,0,0}\makebox(0,0)[lt]{\lineheight{1.25}\smash{\begin{tabular}[t]{l}$\beta_{i_{j+1}}$\end{tabular}}}}%
    \put(0.1254958,0.35363679){\color[rgb]{0,0,0}\makebox(0,0)[lt]{\lineheight{1.25}\smash{\begin{tabular}[t]{l}$\alpha_{i_j}$\end{tabular}}}}%
    \put(0.65684713,0.3511228){\color[rgb]{0,0,0}\makebox(0,0)[lt]{\lineheight{1.25}\smash{\begin{tabular}[t]{l}$\alpha_{i_{j+1}}$\end{tabular}}}}%
    \put(0.12892174,0.18933232){\color[rgb]{0,0,0}\makebox(0,0)[lt]{\lineheight{1.25}\smash{\begin{tabular}[t]{l}$\beta_{j_j}$\end{tabular}}}}%
    \put(0.67647966,0.17932473){\color[rgb]{0,0,0}\makebox(0,0)[lt]{\lineheight{1.25}\smash{\begin{tabular}[t]{l}$\beta_{i_{j+1}}$\end{tabular}}}}%
    \put(0.35102595,0.24385203){\color[rgb]{0,0,0}\makebox(0,0)[lt]{\lineheight{1.25}\smash{\begin{tabular}[t]{l}$\sigma_s$\end{tabular}}}}%
    \put(0.3628673,0.00368848){\color[rgb]{0,0,0}\makebox(0,0)[lt]{\lineheight{1.25}\smash{\begin{tabular}[t]{l}$\gamma'_{s}$\end{tabular}}}}%
    \put(0.04619872,0.12293803){\color[rgb]{0,0,0}\makebox(0,0)[lt]{\lineheight{1.25}\smash{\begin{tabular}[t]{l}$u_{i_j}$\end{tabular}}}}%
    \put(0.25370356,0.12536587){\color[rgb]{0,0,0}\makebox(0,0)[lt]{\lineheight{1.25}\smash{\begin{tabular}[t]{l}$v_{i_j}$\end{tabular}}}}%
    \put(0,0){\includegraphics[width=\unitlength,page=2]{drawing-7.pdf}}%
    \put(0.55155592,0.12342411){\color[rgb]{0,0,0}\makebox(0,0)[lt]{\lineheight{1.25}\smash{\begin{tabular}[t]{l}$u_{i_{j+1}}$\end{tabular}}}}%
    \put(0.7697791,0.11596225){\color[rgb]{0,0,0}\makebox(0,0)[lt]{\lineheight{1.25}\smash{\begin{tabular}[t]{l}$v_{i_{j+1}}$\end{tabular}}}}%
    \put(0.12256338,0.08661569){\color[rgb]{0,0,0}\makebox(0,0)[lt]{\lineheight{1.25}\smash{\begin{tabular}[t]{l}$\alpha_{i_j}$\end{tabular}}}}%
    \put(0.64996596,0.08497157){\color[rgb]{0,0,0}\makebox(0,0)[lt]{\lineheight{1.25}\smash{\begin{tabular}[t]{l}$\alpha_{i_{j+1}}$\end{tabular}}}}%
    \put(-0.00164951,0.30374147){\color[rgb]{0,0,0}\makebox(0,0)[lt]{\lineheight{1.25}\smash{\begin{tabular}[t]{l}$\alpha_{i_j}(0)$\end{tabular}}}}%
    \put(0.1693449,0.28729975){\color[rgb]{0,0,0}\makebox(0,0)[lt]{\lineheight{1.25}\smash{\begin{tabular}[t]{l}$\alpha_{i_j}(1)$\end{tabular}}}}%
    \put(0.60392895,0.29049371){\color[rgb]{0,0,0}\makebox(0,0)[lt]{\lineheight{1.25}\smash{\begin{tabular}[t]{l}$\alpha_{i_{j+1}}(0)$\end{tabular}}}}%
    \put(0.76834677,0.27734029){\color[rgb]{0,0,0}\makebox(0,0)[lt]{\lineheight{1.25}\smash{\begin{tabular}[t]{l}$\alpha_{i_{j+1}}(1)$\end{tabular}}}}%
  \end{picture}%
\endgroup%

%% file: drawing-8.pdf_tex
\begingroup%
  \makeatletter%
  \providecommand\color[2][]{%
    \errmessage{(Inkscape) Color is used for the text in Inkscape, but the package 'color.sty' is not loaded}%
    \renewcommand\color[2][]{}%
  }%
  \providecommand\transparent[1]{%
    \errmessage{(Inkscape) Transparency is used (non-zero) for the text in Inkscape, but the package 'transparent.sty' is not loaded}%
    \renewcommand\transparent[1]{}%
  }%
  \providecommand\rotatebox[2]{#2}%
  \newcommand*\fsize{\dimexpr\f@size pt\relax}%
  \newcommand*\lineheight[1]{\fontsize{\fsize}{#1\fsize}\selectfont}%
  \ifx\svgwidth\undefined%
    \setlength{\unitlength}{735.86297342bp}%
    \ifx\svgscale\undefined%
      \relax%
    \else%
      \setlength{\unitlength}{\unitlength * \real{\svgscale}}%
    \fi%
  \else%
    \setlength{\unitlength}{\svgwidth}%
  \fi%
  \global\let\svgwidth\undefined%
  \global\let\svgscale\undefined%
  \makeatother%
  \begin{picture}(1,0.2750882)%
    \lineheight{1}%
    \setlength\tabcolsep{0pt}%
    \put(0,0){\includegraphics[width=\unitlength,page=1]{drawing-8.pdf}}%
    \put(0.15892759,0.04704873){\color[rgb]{0,0,0}\makebox(0,0)[lt]{\lineheight{1.25}\smash{\begin{tabular}[t]{l}$y_{j-1}$\end{tabular}}}}%
    \put(-0.0012342,0.15916199){\color[rgb]{0,0,0}\makebox(0,0)[lt]{\lineheight{1.25}\smash{\begin{tabular}[t]{l}$x_{j-1}$\end{tabular}}}}%
    \put(0.15164753,0.25817106){\color[rgb]{0,0,0}\makebox(0,0)[lt]{\lineheight{1.25}\smash{\begin{tabular}[t]{l}$z_{j-1}$\end{tabular}}}}%
    \put(0.27104083,0.00336818){\color[rgb]{0,0,0}\makebox(0,0)[lt]{\lineheight{1.25}\smash{\begin{tabular}[t]{l}$y_{j} \in \tilde{B}_{j-1} \times \{t_{j} / \xi_{j}\} \cap \tilde{B}_{j} \times \{t_j\}$\end{tabular}}}}%
    \put(0.31472136,0.26253916){\color[rgb]{0,0,0}\makebox(0,0)[lt]{\lineheight{1.25}\smash{\begin{tabular}[t]{l}$x_{j} \in \tilde{B}_{j-1} \times \{0\} \subset \tilde{S}$\end{tabular}}}}%
    \put(0.49381139,0.06743295){\color[rgb]{0,0,0}\makebox(0,0)[lt]{\lineheight{1.25}\smash{\begin{tabular}[t]{l}$z_{j} \in \tilde{B}_{j} \times \{0\} \subset \tilde{S}$\end{tabular}}}}%
    \put(0.27686495,0.07908106){\color[rgb]{0,0,0}\makebox(0,0)[lt]{\lineheight{1.25}\smash{\begin{tabular}[t]{l}$y'_{j}$\end{tabular}}}}%
    \put(0.22590438,0.0470487){\color[rgb]{0,0,0}\makebox(0,0)[lt]{\lineheight{1.25}\smash{\begin{tabular}[t]{l}$\tilde{\gamma}_{j-1}$\end{tabular}}}}%
  \end{picture}%
\endgroup%